\documentclass[a4paper,12pt]{amsart}
\usepackage[utf8]{inputenc}
\usepackage[T1]{fontenc}
\usepackage[UKenglish]{babel}
\usepackage[a4paper,margin=28mm]{geometry}
\usepackage{verbatim}
\usepackage{xcolor}
\usepackage[round,authoryear]{natbib}
\allowdisplaybreaks[4]
\usepackage{times}
\usepackage{mathrsfs}
\usepackage{amsmath}
\usepackage{amsthm}
\usepackage{amssymb}
\usepackage{amsfonts}
\usepackage{latexsym}
\usepackage{booktabs}
\usepackage{marginnote}
\newcommand{\vp}{\varphi}

\usepackage[colorlinks,linkcolor=red,pagebackref]{hyperref}
\usepackage{xcolor}

\newtheorem{theorem}{Theorem}[section]
\newtheorem{lemma}[theorem]{Lemma}
\newtheorem{proposition}[theorem]{Proposition}
\newtheorem{corollary}[theorem]{Corollary}
\newtheorem{assumption}[theorem]{Assumption}

\theoremstyle{definition}

\theoremstyle{remark}
\newtheorem{remark}[theorem]{Remark}
\makeatletter
\def\th@remark{\itshape}
\makeatother
\numberwithin{equation}{section}
\renewcommand{\labelenumi}{\roman{enumi})}
\renewcommand\theenumi\labelenumi
\renewcommand{\leq}{\leqslant}
\renewcommand{\le}{\leqslant}
\renewcommand{\geq}{\geqslant}
\renewcommand{\ge}{\geqslant}

\newcommand{\Be}{\begin{equation}}
	\newcommand{\Ees}{\end{equation*}}
\newcommand{\Bes}{\begin{equation*}}
	\newcommand{\Ee}{\end{equation}}

\newcommand{\R}{\mathbb{R}}

\newcommand{\E}{\mathbb{E}}
\newcommand{\e}{\varepsilon}

\newcommand{\dif}{\mathrm{d}}

\begin{document}
	\title[Normal approximation for LMC]
	{High-dimensional normal approximations for sums of  Langevin Markov chains}

	\author[T. Shen]{Tian Shen}
\address{Tian Shen: School of Statistics and Data Science, Shanghai University of International Business and Economics, Shanghai, China;}
\email{shentian@suibe.edu.cn}
	\author[Z. G. Su]{Zhonggen Su}
	\address{Zhonggen Su: School of Mathematical Sciences, Zhejiang University, China.}
	\email{suzhonggen@zju.edu.cn}
	
	\author[X. L. Wang ]{Xiaolin Wang}
	\address{Xiaolin Wang: Department of Statistics and Data Science, The Chinese University of Hong Kong, Hong Kong.}
	\email{xiaolinwang@cuhk.edu.hk}


\begin{abstract}
	Consider the well-known Langevin diffusion on $\R^d$
\begin{align*}
    \dif X_t = -\nabla U(X_t)\,\dif t + \sqrt{2}\dif B_t,
\end{align*}
 and its Euler-Maruyama discretization  given by
\begin{align*}
    X_{k+1}=X_k-\eta \nabla U(X_k)+\sqrt{2\eta }\xi_{k+1},
\end{align*}
where $\eta$ is the step size. Under mild conditions, the Langevin diffusion  admits $\pi(\dif x)\propto \exp(-U(x))\dif x$  as its  unique stationary distribution.
 In this paper, we mainly  study the normal approximation  of the normalized partial sum 
\begin{align*}
    W_n = \eta^{1/2} n^{-1/2}  \sum_{i=0}^{n-1} \left(X_i 
 - \int_{\mathbb{R}^d} x\,\pi(\dif x) \right).
\end{align*}
To the best of 
our knowledge, this work provides the first dimension-explicit convergence rates   in high-dimensional settings. Our main tool is a novel upper bound for the 1-Wasserstein distance $\mathcal W_1(W,\gamma)$ via the exchangeable pair approach, where $W$ is any random vector of interest and $\gamma$ is a $d$-dimensional standard normal random vector.
\end{abstract}

	\maketitle
	
	\section{Introduction and Main Results}\label{S1}

This paper studies a quantitative multivariate central limit theorem for ergodic averages generated by the unadjusted Langevin algorithm in high dimension. The problem combines two sources of difficulty. First, the iterates form a strongly dependent Markov chain rather than a sequence of independent observations. Second, the transition kernel and its invariant distribution depend on the step size $\eta$. Thus, when $n\to\infty$, $\eta\to0$, and the dimension $d$ is allowed to vary, a useful normal-approximation bound must track the sampling fluctuation along the numerical trajectory together with the error produced by the Euler discretization. The primary purpose of the present paper is to obtain such a bound with explicit dependence on $d$, $n$, and $\eta$.

\subsection{Langevin Markov chains}

Let $\pi$ have density
$$
    \pi(x)=\frac1Z\exp\{-U(x)\},
    \qquad x\in\mathbb R^d,
$$
on $\mathbb R^d$, where $Z$ is the normalizing constant and $U$ is
a potential function. The overdamped Langevin diffusion associated with
$\pi$ is
\begin{equation}\label{LSDE}
    \dif X_t=-\nabla U(X_t)\,\dif t+\sqrt2\,\dif B_t,
\end{equation}
where $(B_t)_{t\geq0}$ is a standard $d$-dimensional Brownian motion. Under the assumptions imposed below, \eqref{LSDE} admits a unique strong solution, is reversible with respect to $\pi$, and has generator
$$
    \mathcal Af
    =-\langle\nabla U,\nabla f\rangle+\Delta f.
$$

Suppose that $U$ is $\alpha$-strongly convex and $\beta$-smooth, or equivalently,
\begin{equation}\label{2003}
    \alpha I_d\preceq\nabla^2U(x)\preceq\beta I_d,
    \qquad x\in\mathbb R^d.
\end{equation}
Synchronous coupling then yields exponential contraction of the diffusion. In particular, if $\pi_t$ denotes the law of $X_t$, then
\begin{equation}
    \mathcal W_2(\pi_t,\pi)
    \leq
    \mathrm e^{-\alpha t}\mathcal W_2(\pi_0,\pi);
\end{equation}
see, for example, \cite[Theorem~1.4.12]{Chewi2025}. Related quantitative convergence results based on functional inequalities and coupling methods can be found in \cite{BCG2000,BGL2014,CG2009,Eberle2015}.

The Euler--Maruyama discretization of \eqref{LSDE} is
\begin{equation}
    X_{k+1}
    =X_k-\eta_{k+1}\nabla U(X_k)
    +\sqrt{2\eta_{k+1}}\,\xi_{k+1},  \quad k\geq 0,
\end{equation}
where $(\xi_k)_{k\geq1}$ is an i.i.d. sequence of standard Gaussian random vectors. The use of this recursion for sampling goes back to \cite{Parisi1981,G1983,GM1994}; it is commonly called the unadjusted Langevin algorithm (ULA), or Langevin Markov chain (LMC). We consider a constant step size $\eta$, in which case
\begin{equation}\label{LMC-1}
    X_{k+1}
    =X_k-\eta\nabla U(X_k)+\sqrt{2\eta}\,\xi_{k+1}.
\end{equation}
Under suitable conditions, the Markov chain  $\{X_k, k\geq 0\}$  is geometrically ergodic and admits an invariant distribution $\pi_\eta$; see \cite{RT1996}. In general, $\pi_\eta\neq\pi$. Quantitative comparisons between  $\pi_\eta$  and  $\pi$ have been obtained in total variation, weighted total variation, and Wasserstein distances; see, among others, \cite{dalalyan2017theoretical,DM2017,durmus2019analysis,Fang2018}. In particular,  
$\pi_\eta \rightarrow \pi$ as $ \eta \rightarrow 0$.  Concentration properties of $\pi_\eta$ have also been recently studied in \cite{al24}.

The LMC is widely used for approximate sampling from $\pi$ and for
estimating expectations under $\pi$. In many applications, direct
sampling from $\pi$ or exact evaluation of a high-dimensional
integral is infeasible. A common approach is therefore to use the
ergodic average
$$
    \widehat\pi_n(g)
    :=\frac1n\sum_{k=0}^{n-1}g(X_k).
$$
For fixed $\eta>0$, under suitable conditions, as $n\to\infty$,
$$
    \widehat\pi_n(g)\rightarrow
    \pi_\eta(g)
    =\int_{\mathbb R^d}g(x)\pi_\eta(\dif x).
$$
In this paper, we study the normal approximation for the error of this
estimator when $g(x)=x$.
However, in practice, it is important to quantify the approximation error of such estimators. With this motivation in mind, the present paper establishes quantitative bounds for the normal approximation of the normalized estimation error $\sqrt{n\eta}(\widehat\pi_n(g)-\pi(g))$ with $g(x)=x$.

The restriction to the Langevin Markov chains is deliberate.
Under \eqref{2003}, the synchronous coupling of two copies of
\eqref{LMC-1} gives
$$
    |X_{k+1}^x-X_{k+1}^y|^2
    \leq
    \bigl(1-2\alpha\eta+\beta^2\eta^2\bigr)
    |X_k^x-X_k^y|^2.
$$
For $0<\eta<\alpha/(2\beta^2)$, set
$\rho_\eta=(1-2\alpha\eta+\beta^2\eta^2)^{1/2}<1$.
Let $P_\eta$ denote the transition kernel and
$\beta_\eta(k)$ the $\beta$-mixing coefficient of the
stationary chain. For $k\geq1$,
\begin{equation}\label{eq:beta-mixing}
\begin{aligned}
    \beta_\eta(k)
    &=
    \int_{\mathbb R^d}
    \|P_\eta^k(x,\cdot)-\pi_\eta\|_{\mathrm{TV}}
    \,\pi_\eta(\dif x)\leq
    \frac{\rho_\eta^k}{2\sqrt{\pi\eta}}
    \mathbb E|X_0-X_0'|,
\end{aligned}
\end{equation}
where $X_0'$ is an independent copy of $X_0\sim\pi_\eta$ and
$\|\mu-\nu\|_{\mathrm{TV}}=\sup_A|\mu(A)-\nu(A)|$.
The inequality follows from synchronous coupling for the first
$k-1$ steps and a total variation bound for Gaussian
distributions in the last step.
Together with the moment bound in Lemma~\ref{Bound},
this implies geometric $\beta$-mixing.

The Jacobi flow of the Langevin diffusion gives the derivative
bounds needed for the Stein equations. Our proof uses the
pathwise contraction and these derivative bounds to track
$d$, $n$, and $\eta$ in the final estimate.

\subsection{High-dimensional normal approximation}

Let
$$
    m_\pi:=\int_{\mathbb R^d}x\,\pi(x)\dif x,
    \qquad
    \overline X_n:=\frac1n\sum_{k=0}^{n-1}X_k,
$$
and define
\begin{equation}\label{LMC-2}
    W_n
    :=\frac{\sqrt\eta}{\sqrt n}
    \sum_{k=0}^{n-1}(X_k-m_\pi)
    =\sqrt{n\eta}\,(\overline X_n-m_\pi).
\end{equation}
Then $W_n$ is the error of the ULA ergodic-average estimator of the target mean, normalized at the diffusion central-limit scale. Writing $T=n\eta$, we may equivalently express
$$
    W_n
    =\frac1{\sqrt T}
    \sum_{k=0}^{n-1}\eta(X_k-m_\pi),
$$
which is the discrete-time counterpart of the continuous-time additive functional
$$
    \frac1{\sqrt T}\int_0^T(X_t-m_\pi)\,\dif t.
$$
Thus, $n\eta$ is the continuous running time represented by $n$ Euler steps.

In addition, we assume $X_0\sim\pi_\eta$.  This stationary initialization removes the additional burn-in term arising from a nonstationary Markov chain. Let
$$
    m_{\pi_\eta}
    :=\int_{\mathbb R^d}x\,\pi_\eta(\dif x).
$$
Then the statistic in \eqref{LMC-2} admits the exact decomposition
\begin{equation}\label{LMC-22}
    W_n
    =
    \sqrt{\frac\eta n}
    \sum_{k=0}^{n-1}(X_k-m_{\pi_\eta})
    +
    \sqrt{n\eta}\,(m_{\pi_\eta}-m_\pi).
\end{equation}
The first term of the right hand side in \eqref{LMC-22} is the fluctuation around the invariant mean of the Markov chain, while the second term is the deterministic bias caused by replacing $\pi$ with $\pi_\eta$.  To measure the similarity with a normal distribution, we need control of the covariance of the discrete additive functional relative to the continuous-time asymptotic covariance.
The proof therefore has to control the dependence along the chain, the
difference between $\pi_\eta$ and $\pi$, and the covariance error
caused by the discretization.

Before stating the assumptions and the main result, we briefly review
some related work on normal approximation. In one dimension, the classical Berry--Esseen theorem gives the optimal $n^{-1/2}$ rate for the normalized sums of i.i.d. random variables under a finite third-moment condition. For independent random vectors in $\mathbb R^d$, the dependence on the dimension becomes a central issue. If $Y_1,\ldots,Y_n$ are i.i.d. centered random vectors, and satisfy $\operatorname{Cov}(Y_1)=I_d$, then the uniform bound over convex sets, established in the seminal paper \cite{Bentkus2003},  is of the form
$$
    \sup_{A\in\mathcal A}
    \left|
        \mathbb P\left(
            \frac1{\sqrt n}\sum_{k=1}^nY_k\in A
        \right)
        -\gamma(A)
    \right|
    \leq
    \frac{C d^{1/4}\mathbb E|Y_1|^3}{\sqrt n},
$$
where $\mathcal A$ is the class of convex Borel sets and $\gamma$ stands for the $d$-dimensional standard normal distribution.  Some refinements over the class of convex sets were obtained in \cite{Bentkus2004,Raic2019,FK2020}. For hyperrectangles, \cite{CCK2013} initiated a line of high-dimensional Gaussian approximation results with only logarithmic dependence on $d$, at the price of stronger assumptions and a slower dependence on $n$; see also \cite{Lopes2022,KR2020,FK2024}. Note that these set-based metrics are particularly suited to simultaneous inference.

 Besides the uniform distances, Wasserstein distances have been widely used to measure the similarity between two probability distributions; see Subsection \ref{nota} for the mathematical definition of Wasserstein distances. For sums of independent random vectors, multivariate normal approximation in $\mathcal W_1$ (Wasserstein-1 distance) has been studied through several versions of Stein's method. Exchangeable-pair bounds were developed in \cite{ChatterjeeMeckes2008,Fang2022}, while \cite{Fang2018} combined Stein's method with Bismut-type formulas to control solutions of multivariate Stein's equations for Lipschitz test functions. For i.i.d. centered random vectors with a finite third absolute moment, \cite{RaicW1} established the optimal $n^{-1/2}$ dependence in $\mathcal W_1$, removing the logarithmic loss in $n$ present in earlier Lipschitz-test-function bounds. The dimensional factor in such estimates depends on the moment quantities and, in some formulations, logarithmically on $d$. For example, under the isotropic normalization $\operatorname{Cov}(Y_1)=I_d$, one necessarily has $\mathbb E|Y_1|^3\geq d^{3/2}$. Thus, even in the independent setting, an optimal pure power of $d$ cannot be specified without fixing the distributional class and the probability metric.

Stronger high-dimensional normal approximation results have also been obtained. For bounded i.i.d. random vectors, the $\mathcal W_2$ (Wasserstein-2 distance) bound in \cite{Zhai2018} is within a logarithmic factor of the optimal rate in the $(n,d)$-regime. General multivariate bounds in $\mathcal W_p$, together with optimal or improved rates under different moment and regularity assumptions, were developed in \cite{Bonis2020,Bonis2024}. These works identify $n^{-1/2}$ as the natural sample-size rate for independent sums. However, neither of them provides a universal dimension exponent against which  we can make a direct comparison.

Compared with the independent case, normal approximation for dependent observations has been a more complicated and challenging issue. \cite{Merlevede2019} is an excellent book which contains a number of normal approximation results with applications for various dependent structures. Among classical results are qualitative central limit theorems for additive functionals of Markov processes; see \cite{Kipnis1986,Meyn1993,douc}. In the univariate setting, Berry--Esseen theorems with the usual
$n^{-1/2}$ rate are available for strongly or geometrically ergodic
Markov chains under moment and regularity assumptions close to
those in the independent case; see, for example,
\cite{HervePene2010}.

For normal approximation in $\mathcal W_1$,
\citet[Theorem~3.1(b) and Remark~3.3]{DedeckerRio2008}
give an $O(n^{-1/2})$ bound for normalized sums of stationary
centered real-valued random variables under geometric strong
mixing and a finite fourth-moment condition, provided that
the asymptotic variance is positive.
For fixed $d$ and $\eta$, \eqref{eq:beta-mixing} and
Lemma~\ref{Bound} verify these mixing and moment conditions
for $g(X_k)-\pi_\eta(g)$ whenever $g:\mathbb R^d\to\mathbb R$
is Lipschitz.
Here the centering and asymptotic variance are those of the
discrete chain, and the error constant may depend on $d$,
$\eta$, and $g$.

Related results in Wasserstein distances  for dependent sequences include \cite{DedeckerMerlevede2017}, which studies a somewhat different fluctuation problem on empirical distributions. More precisely, they considered the random distance
$$
    \mathcal W_1\left(
        \frac1n\sum_{k=1}^n\delta_{Y_k},
        \mathcal L(Y_0)
    \right)
$$
between the empirical and marginal distributions of a stationary $\alpha$-dependent sequence, and proved moment inequalities and central limit theorems for this empirical-measure functional. 

Results closer to the present statistic $W_n$ have been obtained for Euler schemes and Markov-chain additive functionals. A scalar central limit theorem for ergodic averages of the Euler--Maruyama scheme was proved in \cite{Lu2022}, and normalized and self-normalized moderate deviation results were established in \cite{Fan2024}. More recently, \cite{ZhangXie2026} obtained the optimal $O(n^{-1/2})$ rate in $\mathcal W_1$ for additive functionals of geometrically ergodic Markov chains under suitable drift, moment and nondegeneracy conditions.

For fixed $d$ and $\eta>0$, the result in \cite{ZhangXie2026} can
in principle be applied to the stationary LMC $\{X_k, k\geq0\}$
after verifying the corresponding assumptions. It gives an
$O(n^{-1/2})$ bound in $\mathcal W_1$ for the normal approximation
of $\frac1{\sqrt n}\sum_{k=0}^{n-1}(X_k-m_{\pi_\eta})$, with both
the covariance of the approximating normal vector and the error
constants depending on $d$ and $\eta$. The proof combines the
Stein framework developed in \cite{RaicW1} with forward and
backward maximal couplings.

In this paper, we study the normal approximation for $W_n$ with
explicit dependence on $d$ and $\eta$. Our proof uses an
exchangeable-pair version of Stein's method, where the pair is
constructed by replacing one Gaussian innovation with an
independent copy. Such a replacement may affect all subsequent
iterates. We use the contraction estimate to control these effects
in the exchangeable-pair argument, without imposing a
local-dependence assumption on the summands. This would be helpful for other similar multivariate extensions.

\subsection{Main results}

We start by making the following assumptions.

\begin{assumption}\label{Ass1}
Let $\pi$ have a  density function $Z^{-1}\exp\{-U(x)\}$ on $\mathbb R^d$, and 
assume that the potential function $U(x)$ satisfies the following conditions:

(i) $U\in\mathcal C^5(\mathbb R^d;\R)$;

(ii) $\nabla U(0)=0$;

(iii) $U$ is $\alpha$-strongly convex and $\beta$-smooth, that is,
$$
    \frac\alpha2|x-y|^2
    \leq
    U(y)-U(x)-\langle\nabla U(x),y-x\rangle
    \leq
    \frac\beta2|x-y|^2,
    \qquad x,y\in\mathbb R^d;
$$

(iv) the operator norms of $\nabla^3U(x)$, $\nabla^4U(x)$, and $\nabla^5U(x)$ are uniformly bounded by $M$.
\end{assumption}
We also need the following definition of asymptotic covariance $\Sigma$. For 
$i=1,\ldots,d$, let $h_i(x)=x_i$, where $x=(x_1,\cdots, x_d)$, and let 
$\varphi_i$ solve the Stein equation
$$
    \mathcal A\varphi_i=h_i-\pi(h_i).
$$
Now define $\Sigma=(\Sigma_{ij})_{1\leq i,j\leq d}$, where   
\begin{equation}\label{cm}
    \Sigma_{ij}
    :=2\int_{\mathbb R^d}
    \langle\nabla\varphi_i(x),\nabla\varphi_j(x)\rangle
    \,\pi (x)\dif x.
\end{equation}
Equivalently,
$$
    \Sigma
    =2\int_{\mathbb R^d}
    \nabla\varphi(x)\nabla\varphi(x)^\top\,\pi(x)\dif x,
$$
where
$$
    \nabla\varphi(x)
    :=
    \begin{pmatrix}
        \nabla\varphi_1(x)^\top\\
        \vdots\\
        \nabla\varphi_d(x)^\top
    \end{pmatrix}.
$$
The matrix $\Sigma$ is the asymptotic covariance matrix associated with the coordinate additive functional of the Langevin diffusion. In particular, if $\nabla U(x)=Ax$, where $A$ is symmetric positive definite, then
$$
    \varphi_i(x)=-(A^{-1}x)_i,
    \qquad
    \Sigma=2A^{-2}.
$$
The following lemma further shows that the eigenvalues of $ \Sigma$ are bounded above and below by positive constants independent of $d$, and hence $\Sigma^{-1/2}$ is well defined.
 \begin{lemma}\label{strict}
For the matrix $\Sigma$ defined in \eqref{cm}, if \textbf{Assumption}~\ref{Ass1} holds, then
$$
    \frac{2}{\beta^2}I_d
    \preceq
    \Sigma
    \preceq
    \frac{2}{\alpha^2}I_d.
$$
In particular, $\Sigma$ is positive definite.
\end{lemma}

Throughout the paper,  $\gamma$ denotes the $d$-dimensional standard normal distribution. Now we are ready to state our main result as follows.
\begin{theorem}\label{T2}
Let $U$ satisfy \textbf{Assumption}~\ref{Ass1}, let
$\eta\in(0,\alpha/(2\beta^2))$, and run the Langevin Monte Carlo algorithm with $X_0\sim\pi_\eta$, independently of $(\xi_k)_{k\geq1}$. Define $X_k$ and $W_n$ as in \eqref{LMC-1} and \eqref{LMC-2}, respectively. Then
\begin{align}\label{mainresult}
    \mathcal W_1(\Sigma^{-1/2}W_n,\gamma)
    \leq C\left[
        \frac{d^2}{\sqrt{n\eta}}
        +\frac{d^3\log(\mathrm e\,nd)}{\sqrt n}
        +\sqrt\eta\,d^2
        +\sqrt{n\eta^3}\,d^{5/2}
    \right],
\end{align}
where $\Sigma$ is the positive definite matrix defined in \eqref{cm}.
\end{theorem}
As a direct consequence, we have the following corollaries. 
\begin{corollary}\label{Cor1}
Under the assumptions and notation of Theorem~\ref{T2}, suppose that
$n=\lfloor\eta^{-p}\rfloor$ for some $1<p<3$. Then
\begin{align*}
    \mathcal W_1(\Sigma^{-1/2}W_n,\gamma)
    \leq C\bigg[
        &d^2n^{\frac{1}{2p}-\frac12}
        +d^3n^{-1/2}\log(\mathrm e\,nd)
        +d^2n^{-\frac{1}{2p}}\\
        &+d^{5/2}n^{\frac12-\frac{3}{2p}}
    \bigg].
\end{align*}
\end{corollary}

\begin{corollary}\label{rem:sampling-meaning}
Let $\varepsilon_{n,\eta,d}$ denote the right-hand side of \eqref{mainresult}. Then for every $u\in\mathbb R^d\setminus\{0\}$,
$$
    \mathcal W_1\left(
        \frac{\sqrt{n\eta}\,u^\top(\overline X_n-m_\pi)}
        {\sqrt{u^\top\Sigma u}},
        \mathcal N(0,1)
    \right)
    \leq\varepsilon_{n,\eta,d}.
$$
Moreover,
$$
    \mathcal W_1\left(
        \overline X_n,
        \mathcal N\left(m_\pi,\frac\Sigma{n\eta}\right)
    \right)
    \leq
    \frac{\|\Sigma^{1/2}\|_{\mathrm{op}}}{\sqrt{n\eta}}
    \varepsilon_{n,\eta,d}.
$$
Thus,  we have a quantitative Gaussian approximation for the average Monte Carlo error of the ULA estimator of the target mean.
\end{corollary}
\begin{proof}
Recall from \eqref{LMC-2} that $W_n
    =
    \sqrt{n\eta}\,(\overline X_n-m_\pi)$ and let $Y_n:=\Sigma^{-1/2}W_n.$ Then Theorem~\ref{T2} easily gives $\mathcal W_1(Y_n,\gamma)
    \leq
    \varepsilon_{n,\eta,d}.$ Fix $u\in\mathbb R^d\setminus\{0\}$, and define
$$
    a_u
    :=
    \frac{\Sigma^{1/2}u}
    {\sqrt{u^\top\Sigma u}}.
$$
Since $|a_u|^2
    =
   1,$ the map
$$
    T_u(x):=a_u^\top x,
    \qquad x\in\mathbb R^d,
$$
is $1$-Lipschitz. In addition, if $Z\sim\gamma$, then
$
    T_u(Z)\sim\mathcal N(0,1).
$
The contraction property of the Wasserstein distance under Lipschitz
maps therefore yields
\begin{align*}
    &\mathcal W_1\left(
        \frac{\sqrt{n\eta}\,
        u^\top(\overline X_n-m_\pi)}
        {\sqrt{u^\top\Sigma u}},
        \mathcal N(0,1)
    \right)
    =
    \mathcal W_1\bigl(T_u(Y_n),T_u(Z)\bigr)
    \leq
    \mathcal W_1(Y_n,Z)
    \leq
    \varepsilon_{n,\eta,d},
\end{align*}
which proves the first assertion.

For the second assertion, define
$$
    T(x)
    :=
    m_\pi+\frac{1}{\sqrt{n\eta}}\Sigma^{1/2}x,
    \qquad x\in\mathbb R^d.
$$
Then the Lipschitz constant of $T$ is
$$
    \|T\|_{\mathrm{Lip}}
    =
    \frac{\|\Sigma^{1/2}\|_{\mathrm{op}}}
    {\sqrt{n\eta}}.
$$
Hence, again by the contraction property of
$\mathcal W_1$, we have
\begin{align*}
    \mathcal W_1\left(
        \overline X_n,
        \mathcal N\left(
            m_\pi,\frac{\Sigma}{n\eta}
        \right)
    \right)
    &=
    \mathcal W_1\bigl(T(Y_n),T(Z)\bigr)\leq
    \frac{\|\Sigma^{1/2}\|_{\mathrm{op}}}
    {\sqrt{n\eta}}
    \mathcal W_1(Y_n,Z)\leq
    \frac{\|\Sigma^{1/2}\|_{\mathrm{op}}}
    {\sqrt{n\eta}}
    \varepsilon_{n,\eta,d}.
\end{align*}
We now complete the proof.
\end{proof}

 The proof of Theorem~\ref{T2} is postponed to Section \ref{S3}. Now we would like to outline the proof strategy. Firstly, we use the Stein equation associated with the Langevin diffusion to decompose $W_n$ into a martingale term $\mathcal H_n$ and an Euler discretization remainder $\mathcal R_n$, namely 
 $$
 W_n=\mathcal H_n+\mathcal R_n.
 $$
 
 Secondly, we control $\mathcal W_1(\Sigma^{-1/2}\mathcal H_n,\gamma)$ by using  a novel multivariate exchangeable-pair version of Stein's method, where the pair is constructed by replacing one Gaussian term with an independent copy. To control the conditional covariance and perturbation terms, we need to use the explicit contractivity property of the Euler scheme and regularity estimates for the Stein equation associated with the Langevin generator $\mathcal A$.  In turn, Assumption~\ref{Ass1} (iv) is used to obtain the higher-order Jacobi-flow estimates which are required for the regularity estimates.

 Thirdly, the remainder $\mathcal R_n$ is estimated directly by martingale arguments, Gaussian moment bounds, and stationary moment estimates. Finally, combining the two bounds  with the triangle inequality for $\mathcal W_1$ concludes the proof of the theorem.
 
\begin{remark} (1)  Theorem~\ref{T2}  gives a nonasymptotic normal approximation in 
$\mathcal W_1$ for the vector-valued coordinate observables. This can be viewed as a complement of  the scalar asymptotic result in \cite{Lu2022} and the moderate deviation result  in \cite{Fan2024}.

 (2) We  state the explicit dependence on the parameters $d, n, \eta$ in Theorem~\ref{T2}, but we do not claim that the dimension dependence in Theorem~\ref{T2} is optimal. Indeed, the substantially sharper dependence will be obtained in the linear Gaussian case in Section~\ref{S2}. This implies that part of the dimensional loss in the nonlinear case may come from the Stein equation and discretization-remainder estimates. 

(3) If $n=\lfloor\eta^{-p}\rfloor$ with $1<p<3$, then for fixed
$d$, the three $\eta$-dependent terms in Theorem~\ref{T2} are
balanced optimally at $p=2$. In this case,
$$
    \mathcal W_1(\Sigma^{-1/2}W_n,\gamma)
    \leq
    C\left[
        d^{5/2}n^{-1/4}
        +d^3n^{-1/2}\log(\mathrm e\,nd)
    \right].
$$
Thus, for fixed $d$, the best rate given by our bound over
$1<p<3$ is $n^{-1/4}$. 
\end{remark}

\subsection{Notations}\label{nota}

We conclude this section by collecting the notation used throughout the paper. For any $p\geq1$, let $\mathcal P_p(\mathbb R^d)$ be the set of probability measures on $\mathbb R^d$, equipped with the Euclidean norm, that have finite $p$-th moment. The $p$-Wasserstein distance between $\mu,\nu\in\mathcal P_p(\mathbb R^d)$ is defined by
$$
    \mathcal W_p(\mu,\nu)
    :=\inf_{\rho\in\Pi(\mu,\nu)}
    \left(
        \int_{\mathbb R^d\times\mathbb R^d}
        |x-y|^p\,\rho(\dif x,\dif y)
    \right)^{1/p},
$$
where $\Pi(\mu,\nu)$ is the set of couplings of $\mu$ and $\nu$. For any random vector $Y$, $\mathcal L(Y)$ denotes its distribution.  For a random vector $Y$ and a probability measure $\nu$, we write
$$
    \mathcal W_p(Y,\nu)
    :=\mathcal W_p(\mathcal L(Y),\nu).
$$
Likewise, for random vectors $Y$ and $Z$,
$$
    \mathcal W_p(Y,Z)
    :=\mathcal W_p(\mathcal L(Y),\mathcal L(Z)).
$$
In particular, the Kantorovich dual formula gives
$$
    \mathcal W_1(\mu,\nu)
    =\sup_{f: \ \|f\|_{\mathrm{Lip}}\leq1}|\mu(f)-\nu(f)|;
$$
see \cite[Remark~6.5]{Villani2009}.

For vectors $x,y\in\mathbb R^d$, let $\langle x,y\rangle$ denote the Euclidean inner product and let $|x|$ denote the Euclidean norm. For matrices $A,B\in\mathbb R^{d\times d}$, define the Hilbert--Schmidt inner product by
$$
    \langle A,B\rangle_{\mathrm{HS}}
    :=\sum_{i,j=1}^dA_{ij}B_{ij}.
$$
The Hilbert--Schmidt norm and operator norm of $A$ are defined by
	$$
	\|A\|_{\mathrm{HS}} := \Bigl(\sum_{i,j=1}^d A_{ij}^2\Bigr)^{1/2},
	\qquad
	\|A\|_{\mathrm{op}} := \sup_{|x|=1}|Ax|
	= \sup_{\substack{|u_1|=|u_2|=1}}
	\bigl|\langle A, u_1u_2^{\top}\rangle_{\mathrm{HS}}\bigr|.
	$$
	We have the following relations:
	\begin{equation}\label{ophs}
		\|A\|_{\mathrm{op}}
		= \sup_{\substack{|u_1|=|u_2|=1}}
		\bigl|\langle A, u_1u_2^{\top}\rangle_{\mathrm{HS}}\bigr|,
		\qquad
		\|A\|_{\mathrm{op}} \le \|A\|_{\mathrm{HS}} \le \sqrt{d}\,\|A\|_{\mathrm{op}}.
	\end{equation}
	For order-$r$ tensors $A,B\in(\mathbb{R}^d)^{\otimes r}$, the Hilbert--Schmidt inner product is defined by
	$$
	\langle A,B\rangle_{\mathrm{HS}}
	:= \sum_{i_1=1}^d \cdots \sum_{i_r=1}^d
	A_{i_1,\dots,i_r}\,B_{i_1,\dots,i_r}.
	$$
	
	Let $\mathcal{C}^k(\mathbb{R}^d,\mathbb{R})$ denote the space of $k$-times continuously differentiable functions from $\mathbb{R}^d$ to $\mathbb{R}$. For $f\in\mathcal{C}^2(\mathbb{R}^d,\mathbb{R})$, let $\nabla f(x)\in\R^d$, $\nabla^2 f(x)\in\R^{d\times d}$ and $\Delta f(x)\in\R$ denote the gradient, the Hessian matrix and the Laplacian of $f$, respectively.

	More generally, for any integer $r\ge 3$, the $r$-th derivative $\nabla^r f(x)$ is a symmetric $r$-linear map from $(\mathbb{R}^d)^r$ to $\mathbb{R}$. It can be represented as an order-$r$ tensor with entries
	$$
	\bigl(\nabla^r f(x)\bigr)_{i_1,\dots,i_r}
	= \frac{\partial^r f(x)}{\partial x_{i_1}\cdots \partial x_{i_r}},
	\qquad 1\le i_1,\dots,i_r\le d.
	$$
	For vectors $v_1,\dots,v_r\in\mathbb{R}^d$, its evaluation is defined by
	\begin{align*}
		\nabla^r f(x)[v_1,\dots,v_r]
		&:= \frac{\partial^r f}{\partial v_1 \partial v_2 \cdots \partial v_r}(x) \\
		&= \sum_{i_1=1}^d \cdots \sum_{i_r=1}^d
		\frac{\partial^r f(x)}{\partial x_{i_1}\cdots \partial x_{i_r}}
		\,(v_1)_{i_1}\cdots (v_r)_{i_r}.
	\end{align*}
	
	The operator norm of $\nabla^r f(x)$ is
	$$
	\|\nabla^r f(x)\|_{\mathrm{op}}
	:= \sup_{\substack{|v_1|=\cdots=|v_r|=1}}
	\left| \nabla^r f(x)[v_1,\dots,v_r] \right|.
	$$
	In the special case $r=2$, $\|\nabla^2 f(x)\|_{\mathrm{op}}$ coincides with the spectral norm of the Hessian matrix, i.e., the largest absolute eigenvalue of $\nabla^2 f(x)$.

	In particular, if $A=\nabla^r f(x)$ and $B=v_1\otimes\cdots\otimes v_r$ for vectors $v_1,\dots,v_r\in\mathbb{R}^d$, then
	$$
	\langle \nabla^r f(x),\, v_1\otimes\cdots\otimes v_r\rangle_{\mathrm{HS}}
	= \sum_{i_1=1}^d \cdots \sum_{i_r=1}^d
	\frac{\partial^r f(x)}{\partial x_{i_1}\cdots \partial x_{i_r}}
	\prod_{j=1}^r (v_j)_{i_j}.
	$$
	
	For symmetric matrices $A,B\in\mathbb{R}^{d\times d}$, we write $A\succ B$ (resp.\ $A\prec B$) if $A-B$ (resp.\ $B-A$) is positive definite, and $A\succeq B$ (resp.\ $A\preceq B$) if $A-B$ (resp.\ $B-A$) is positive semidefinite. We denote by $I_d$ the $d\times d$ identity matrix.
	
	Throughout the paper, $\gamma$ denotes the distribution of a $d$-dimensional standard normal random vector and $C$ denotes some positive constant depending only on $\alpha$, $\beta$, and $M$, whose value may change from line to line.
	
	The paper is organized as follows. In Section \ref{S2}, we focus on a special case where $\nabla U(x)=Ax$, with $A$ being a positive definite matrix. In this setting, we establish a convergence rate that is faster than the rate obtained in the general case. Section \ref{S3} is devoted to the proof of our main result, Theorem \ref{T2}. We decompose the target $W_n$ into two parts: a martingale component and a remainder term. First, we develop Stein's method to obtain an estimate between  the martingale part and $\gamma$. Next, we bound the remainder term via direct stochastic calculations. At the end of Section \ref{S3}, we complete the  proof  by combining the estimates from these two parts. Appendix \ref{B} is devoted to proofs of auxiliary results stated in the previous Sections.

\section{A Warm-up--The linear case}\label{S2}

As a warm-up, we first consider the linear case
$\nabla U(x)=Ax$, where $A$ is symmetric positive definite.
In this case, $m_\pi=0$ and the continuous-time asymptotic
covariance defined in \eqref{cm} is
$
    \Sigma=2A^{-2}.
$

\begin{theorem}\label{T1}
Suppose that $\alpha I_d\preceq A\preceq\beta I_d.$ Let $p>1$ and set
$
\eta_0:=\min\left\{\frac12,\frac1{2\beta}\right\}.
$ Assume $0<\eta\leq\eta_0$,
and  that $X_0$ is independent of
$(\xi_k)_{k\geq1}$, with
$ \E X_0=0,  \E|X_0|^2\leq C_0d.
$ Then,
\begin{align}\label{linear-bound-eta}
    \mathcal W_2(\Sigma^{-1/2}W_n,\gamma)
    \leq
    C\left\{
        \frac{d^{3/2}}{n\eta}
        +\frac{d^{1/2}}{\sqrt{n\eta}}
    \right\}.
\end{align}
Consequently,  if $n=\lfloor\eta^{-p}\rfloor$,
\begin{align}\label{1412gq}
    \mathcal W_2(\Sigma^{-1/2}W_n,\gamma)
    \leq
    C\left\{
        n^{\frac1p-1}d^{3/2}
        +n^{\frac1{2p}-\frac12}d^{1/2}
    \right\}.
\end{align}
In particular, if
$X_0\sim\mathcal N(0,\Sigma_0)$ with $\Sigma_0\preceq C_0I_d$,
then
\begin{align}\label{linear-gaussian-bound-eta}
    \mathcal W_2(\Sigma^{-1/2}W_n,\gamma)
    \leq
    C\frac{d^{1/2}}{n\eta},
\end{align}
and under the same scaling $n=\lfloor\eta^{-p}\rfloor$,
\begin{align}\label{1413gq}
    \mathcal W_2(\Sigma^{-1/2}W_n,\gamma)
    \leq
    Cn^{\frac1p-1}d^{1/2}.
\end{align}
Here, $C$ depends only on $\alpha$, $\beta$, and $C_0$.
\end{theorem}

\begin{proof}
Set
$$
    B:=I_d-\eta A.
$$
The step-size assumption implies
$$
    0\preceq B\preceq(1-\alpha\eta)I_d.
$$
In particular, for every $r\geq0$,
\begin{equation}\label{linear-B-contraction}
    \|B^r\|_{\mathrm{op}}\leq  \|B\|_{\mathrm{op}}^r
    \leq(1-\alpha\eta)^r,
    \qquad
    \|I_d-B^r\|_{\mathrm{op}}\leq1.
\end{equation}
The linear recursion admits the representation
$$
    X_k
    =B^kX_0+\sqrt{2\eta}\sum_{j=1}^kB^{k-j}\xi_j,
    \qquad k\geq1.
$$
Since $m_\pi=0$, summing this identity over
$k=0,\ldots,n-1$ gives
\begin{equation}\label{linear-W-representation}
    W_n=\frac1{\sqrt n}\sum_{k=0}^{n-1}Z_k,
\end{equation}
where
\begin{align}
    Z_0
    &=\eta^{-1/2}A^{-1}(I_d-B^n)X_0,
    \\
    Z_k
    &=\sqrt2\,A^{-1}(I_d-B^{n-k})\xi_k,
    \qquad 1\leq k\leq n-1.
\end{align}

Write
$$
    \Sigma_0:=\operatorname{Cov}(X_0),
    \qquad
    \Sigma_n:=\operatorname{Cov}(W_n).
$$
By independence and \eqref{linear-W-representation},
\begin{align*}
    \Sigma_n
    ={}&
    \frac1{n\eta}
    A^{-1}(I_d-B^n)\Sigma_0(I_d-B^n)A^{-1}
    \notag\\
    &+\frac2n\sum_{k=1}^{n-1}
    A^{-1}(I_d-B^{n-k})^2A^{-1}.
\end{align*}
Expanding the square in the second term and using
$\Sigma=2A^{-2}$, we obtain
\begin{align}
    \Sigma_n-\Sigma
    ={}&
    \frac1{n\eta}
    A^{-1}(I_d-B^n)\Sigma_0(I_d-B^n)A^{-1}
    -\frac2nA^{-2}
    \notag\\
    &-\frac4n\sum_{r=1}^{n-1}A^{-1}B^rA^{-1}
    +\frac2n\sum_{r=1}^{n-1}A^{-1}B^{2r}A^{-1}.
    \label{linear-covariance-decomposition}
\end{align}
Since
$$
    \|\Sigma_0\|_{\mathrm{op}}
    \leq\E|X_0|^2
    \leq C_0d,
$$
relations \eqref{linear-B-contraction} and
$$
    \sum_{r=1}^{n-1}(1-\alpha\eta)^r
    \leq\frac1{\alpha\eta}
$$
yield
\begin{equation}\label{linear-covariance-general}
    \|\Sigma_n-\Sigma\|_{\mathrm{op}}
    \leq C\frac{d}{n\eta}.
\end{equation}
Here we also used $d\geq1$ and
$n^{-1}\leq(n\eta)^{-1}$, since $\eta\leq1$.

We next Gaussianize only the initial value. Let
$$
    X_0^{\mathrm G}\sim\mathcal N(0,\Sigma_0)
$$
be independent of $X_0$ and $(\xi_k)_{k\geq1}$, and let
$W_n^{\mathrm G}$ be obtained from
\eqref{linear-W-representation} by replacing $X_0$ with
$X_0^{\mathrm G}$. Then $W_n^{\mathrm G}$ is centered Gaussian and
$$
    \operatorname{Cov}(W_n^{\mathrm G})=\Sigma_n.
$$
Moreover,
$$
    W_n-W_n^{\mathrm G}
    =\frac1{\sqrt{n\eta}}
    A^{-1}(I_d-B^n)(X_0-X_0^{\mathrm G}).
$$
Since $\Sigma^{-1/2}=2^{-1/2}A$,
\begin{align}
    \mathcal W_2\left(
        \Sigma^{-1/2}W_n,
        \Sigma^{-1/2}W_n^{\mathrm G}
    \right)
    &\leq
    \left[
        \E\left|
        \Sigma^{-1/2}(W_n-W_n^{\mathrm G})
        \right|^2
    \right]^{1/2}
    \notag\\
    &\leq
    \frac1{\sqrt{2n\eta}}
    \left[\E|X_0-X_0^{\mathrm G}|^2\right]^{1/2}
    \notag\\
    &\leq C\frac{d^{1/2}}{\sqrt{n\eta}}.
    \label{linear-initial-coupling}
\end{align}

It remains to compare the Gaussian vector
$\Sigma^{-1/2}W_n^{\mathrm G}$ with $\gamma$. Define
$$
    S_n=\Sigma^{-1/2}\Sigma_n\Sigma^{-1/2}.
$$
Since $\Sigma_n$ is a covariance matrix, $S_n$ is symmetric and positive semidefinite. Moreover,
$$
    \Sigma^{-1/2}W_n^{\mathrm G}
    \sim\mathcal N(0,S_n).
$$ 
If $G\sim\mathcal N(0,I_d)$, coupling $S_n^{1/2}G$ and $G$
gives
\begin{align}
    \mathcal W_2\bigl(\mathcal N(0,S_n),\gamma\bigr)
    &\leq
    \left[\E\left|(S_n^{1/2}-I_d)G\right|^2\right]^{1/2}
    \notag\\
    &=\|S_n^{1/2}-I_d\|_{\mathrm{HS}}.
    \label{linear-Gaussian-coupling}
\end{align}
By the spectral theorem and the scalar inequality
$|\sqrt{x}-1|\leq|x-1|$ for $x\geq0$,
$$
    \|S_n^{1/2}-I_d\|_{\mathrm{HS}}
    \leq\|S_n-I_d\|_{\mathrm{HS}}.
$$
Since
$$
    S_n-I_d
    =\Sigma^{-1/2}(\Sigma_n-\Sigma)\Sigma^{-1/2},
$$
we obtain from \eqref{linear-covariance-general}
\begin{align}
    \|S_n-I_d\|_{\mathrm{HS}}
    &\leq
    \sqrt d\,\|\Sigma^{-1/2}\|_{\mathrm{op}}^2
    \|\Sigma_n-\Sigma\|_{\mathrm{op}}
    \notag\\
    &\leq C\frac{d^{3/2}}{n\eta}.
    \label{linear-Gaussian-covariance-error}
\end{align}
Combining \eqref{linear-initial-coupling},
\eqref{linear-Gaussian-coupling}, and
\eqref{linear-Gaussian-covariance-error} proves
\eqref{linear-bound-eta}.

Now suppose that $X_0$ itself is Gaussian and that
$\Sigma_0\preceq C_0I_d$. In this case $W_n$ is Gaussian, and
the first term on the right-hand side of
\eqref{linear-covariance-decomposition} is bounded in operator norm by
$C/(n\eta)$. Consequently,
\begin{equation}
    \|\Sigma_n-\Sigma\|_{\mathrm{op}}
    \leq\frac{C}{n\eta}.
\end{equation}
Repeating the Gaussian covariance comparison above, now directly for
$W_n$, yields
$$
    \mathcal W_2(\Sigma^{-1/2}W_n,\gamma)
    \leq C\sqrt d\,\|\Sigma_n-\Sigma\|_{\mathrm{op}}
    \leq C\frac{d^{1/2}}{n\eta},
$$
which proves \eqref{linear-gaussian-bound-eta}.

Finally, since $n=\lfloor\eta^{-p}\rfloor$,
$$
    \eta^{-p}<n+1\leq2n,
$$
and therefore
$$
    \frac1{n\eta}\leq2n^{\frac1p-1},
    \qquad
    \frac1{\sqrt{n\eta}}
    \leq\sqrt2\,n^{\frac1{2p}-\frac12}.
$$
Substituting these estimates into \eqref{linear-bound-eta} and
\eqref{linear-gaussian-bound-eta} gives \eqref{1412gq} and
\eqref{1413gq}, respectively.
\end{proof}

\begin{remark}
The representation \eqref{linear-W-representation} shows that, for a
general centered initial distribution, $W_n$ is the sum of one
possibly non-Gaussian initial term and independent Gaussian innovation
terms. Thus, the non-Gaussian approximation error in the linear model
is entirely due to the initial value. If $X_0$ is Gaussian, then
$W_n$ is itself Gaussian, and only the difference between its
finite-sample covariance $\Sigma_n$ and the limiting covariance
$\Sigma$ has to be controlled.
\end{remark}

	\section{Stein's method and Proof of Theorem \ref{T2}}\label{S3}
    When $\nabla U$ is nonlinear, $W_n$  no longer admits a direct representation as a weighted sum of independent variables, but it can be expressed as the sum of a martingale and a remainder term (see \eqref{ts1} below). Therefore, the strategy of proving Theorem \ref{T2} is to estimate the $1$-Wasserstein distance between this martingale and $\gamma$, as well as  the remainder term.  For the former, we employ Stein's method,  a powerful tool for  studying  such a non-asymptotic error estimate. For background on Stein's method, we refer the reader   to \cite{Stein1972,arratia1989two,chen2011normal,ross2011fundamentals}, and for recent applications, to  \cite{ganguly2019,fang2024}.

In Subsection \ref{sub31}, we first establish a general theorem by Stein's method,  and then employ it to derive an error estimate between the martingale part and $\gamma$. In Subsection \ref{sub32}, an estimate for the remainder term is obtained through direct stochastic calculations. Based on the  results from the preceding two subsections, we shall prove  Theorem \ref{T2} in Subsection \ref{sub33}.

    \subsection{Stein's method}\label{sub31}

    When using Stein's method to bound the distance between a random vector  $W$ of interest and $\gamma$, one typically follows the following three steps.  The first step is to find an appropriate operator ${\mathcal L}$ which satisfies
	\begin{align*}
		\mathbb{E}_{X\sim \gamma} {\mathcal L} g(X)=0
	\end{align*}
	for a class of functions $g$. The second step is to find the solution $g_f$ of the Stein equation 
	\begin{align*}
		{\mathcal L} g(x)=f(x)-\gamma(f) 
	\end{align*}
	and to characterize the properties of $g_f$.  The third step is to express the Wasserstein distance $\mathcal{W}_1(W, \gamma)$ as
	\begin{align*}
		\mathcal{W}_1(W, \gamma)&=\sup_{f:\|f\|_{\rm{Lip}}\leq 1}|\mathbb{E}f(W)-\gamma(f) |\\
        &=\sup_{f:\|f\|_{\rm{Lip}}\leq 1}|\mathbb{E}{\mathcal L} g_f(W) |.
	\end{align*}
In turn, to control the term $|\mathbb{E}{\mathcal L} g_f(W) |$, we use the exchangeable pair approach.

For a random vector $W$ of interest, we make the following assumption.
\begin{assumption}\label{1544gl}
Let $X$ be a random element taking values in a measurable space
$\mathcal X$, and let $f:\mathcal X\to\mathbb R^d$ be measurable with
$W=f(X)$ and $\mathbb E|W|<\infty$. Suppose that the following conditions
hold.

(i) There exists an exchangeable pair $(X,X')$. Writing
$$
    W'=f(X'),\qquad \delta=W'-W,
$$
there exists a measurable antisymmetric function
$D=D(X,X')\in\mathbb R^d$, that is,
$$
    D(X,X')=-D(X',X).
$$

(ii) There exist a scalar $\lambda>0$ and a
$\sigma(X)$-measurable random matrix $\Xi$ such that, almost surely,
\begin{equation}\label{ASSU}
    \mathbb E[D\mid X]= -\lambda W,
    \qquad
    \mathbb E[D\delta^\top\mid X]=2\lambda(I_d+\Xi).
\end{equation}

(iii) The quantities appearing above satisfy
$$
    \mathbb E|D|<\infty,
    \qquad
    \mathbb E\bigl[|D|\,|\delta|\bigr]<\infty,
    \qquad
    \mathbb E\|\Xi\|_{\mathrm{HS}}<\infty,
$$
and
$$
    \mathbb E\left[
        |D|\,|\delta|^2
        \bigl(|\log|\delta||\vee1\bigr)
    \right]<\infty,
$$
where the integrand in the last display is defined to be zero on
$\{\delta=0\}$.
\end{assumption}

	With some modifications to \cite[Theorem~2.5]{Fang2018}, \cite[Theorem~4.1]{Zhang2022} and \cite[Theorem~7.1]{fang2023}, we have the following theorem.
    
	\begin{theorem}\label{L1}
Under \textbf{Assumption}~\ref{1544gl}, we have
$$
    \mathcal W_1(W,\gamma)
    \leq
    C\left\{
        \frac1\lambda
        \mathbb E\left[
            |D|\,|\delta|^2
            \bigl(|\log|\delta||\vee1\bigr)
        \right]
        +
        \mathbb E\|\Xi\|_{\mathrm{HS}}
    \right\}.
$$
\end{theorem}
\begin{remark}
		{In fact, the approach in \cite{Fang2018} is based on the linear conditions
		$$
		\mathbb{E}[\delta \mid X] = -\lambda (W + R_1),
		\quad \mathbb{E}[\delta \delta^\top  \mid X] = 2\lambda (I_d + R_2),
		$$
		and proceeds from the identity
		$$
		0 = \mathbb{E}\big[ f(W') - f(W) \big].
		$$
		A Taylor expansion of the right-hand side  yields terms involving powers of $\delta$. In contrast, our method starts from the identity
		$$
		0 = \mathbb{E} \big\langle D, \nabla f(W) + \nabla f(W') \big\rangle,
		$$
		so that, after Taylor expansion, in each derivative term of $f$ one occurrence of $\delta$ is replaced by $D$. This modification is motivated by the fact that $W$ is globally dependent, and $\delta = W' - W$ may be highly complex. By constructing $D$ such that $\mathbb{E}[D \mid X] = -\lambda W$, we obtain a structure that is simpler than the corresponding $\mathbb{E}[\delta \mid X] = -\lambda(W + R_1)$ in \cite{Fang2018}. Moreover, in higher orders, the magnitude of $|D \,\delta^{\,k-1}|$ is often smaller than that of $|\delta^k|$, which will yield tighter bounds in our setting.}
	\end{remark}
	
\begin{proof}[Proof of Theorem \ref{L1}]
		Taking expectations in the first identity in \eqref{ASSU} and using
		exchangeability and antisymmetry gives $\mathbb E W=0$.
		Since $(W,W')$ is an exchangeable pair and $D$ is
	antisymmetric, for any sufficiently regular function $f$, we have
	\begin{align*}
		0
		&=\mathbb{E}\langle D,\nabla f(W)+\nabla f(W')\rangle\\
		&=2\mathbb{E}\langle D,\nabla f(W)\rangle
		+\mathbb{E}\langle D,\nabla f(W')-\nabla f(W)\rangle\\
		&=-2\lambda\mathbb{E}\langle W,\nabla f(W)\rangle
		+\int_0^1
		\mathbb{E}\langle
		D\delta^\top,\nabla^2f(W+t\delta)
		\rangle_{\mathrm{HS}}\,\dif t,
	\end{align*}
	where we used $\delta=W'-W$ and
	$\mathbb{E}[D\mid X]=-\lambda W$. Thus,
	\begin{align}\label{1116gq}
		\mathbb{E}\langle W,\nabla f(W)\rangle
		=
		\frac{1}{2\lambda}
		\int_0^1
		\mathbb{E}\langle
		D\delta^\top,\nabla^2f(W+t\delta)
		\rangle_{\mathrm{HS}}\,\dif t.
	\end{align}
	
	Next, decompose the integral as
	$$
	\int_0^1
	\mathbb{E}\langle
	D\delta^\top,\nabla^2f(W+t\delta)
	\rangle_{\mathrm{HS}}\,\dif t
	=
	\mathbb{E}\langle
	D\delta^\top,\nabla^2f(W)
	\rangle_{\mathrm{HS}}
	+R,
	$$
	where
	$$
	R=
	\int_0^1
	\mathbb{E}\langle
	D\delta^\top,
	\nabla^2f(W+t\delta)-\nabla^2f(W)
	\rangle_{\mathrm{HS}}\,\dif t.
	$$
	Since $W$ is measurable with respect to $X$, condition
	\eqref{ASSU} gives
	\begin{align}\label{1115gq}
		&\int_0^1
		\mathbb{E}\langle
		D\delta^\top,\nabla^2f(W+t\delta)
		\rangle_{\mathrm{HS}}\,\dif t\notag\\
		&\qquad=
		2\lambda\mathbb{E}\Delta f(W)
		+
		2\lambda\mathbb{E}\langle
		\Xi,\nabla^2f(W)
		\rangle_{\mathrm{HS}}
		+R.
	\end{align}
	Combining \eqref{1116gq} and \eqref{1115gq}, we obtain
	$$
	\mathbb{E}\bigl(
	\Delta f(W)-\langle W,\nabla f(W)\rangle
	\bigr)
	=
	-\mathbb{E}\langle
	\Xi,\nabla^2f(W)
	\rangle_{\mathrm{HS}}
	-\frac{1}{2\lambda}R.
	$$
	
	The Stein operator of the $d$-dimensional standard normal
	distribution $\gamma$ is
	$$
	\mathcal Lf(x)
	=
	\Delta f(x)-\langle x,\nabla f(x)\rangle.
	$$
	Let $h:\mathbb{R}^d\to\mathbb{R}$ be Lipschitz, and let $f$
	be the solution of the standard Gaussian Stein equation
	\begin{align}\label{1128gq}
		\Delta f(x)-\langle x,\nabla f(x)\rangle
		=
		h(x)-\gamma(h).
	\end{align}
	The standard Gaussian Stein solution satisfies
	$$
	\sup_{x\in\mathbb{R}^d}
	\|\nabla^2f(x)\|_{\mathrm{HS}}
	\leq
	C\|h\|_{\mathrm{Lip}},
	$$
	and
	$$
	\|\nabla^2f(x+\varepsilon u)-\nabla^2f(x)\|_{\mathrm{op}}
	\leq
	C\|h\|_{\mathrm{Lip}}
	|\varepsilon|
	\bigl(|\log|\varepsilon||\lor1\bigr),
	$$
	for all $x,u\in\mathbb{R}^d$ with $|u|\leq1$ and
	$\varepsilon\in\mathbb{R}$; see
	\cite[Lemma~2.2(i)]{ChatterjeeMeckes2008} and
	\cite[Theorem~3.1]{Fang2018}.
	
	It follows that
	\begin{align*}
		\left|
		\mathbb{E}\langle
		\Xi,\nabla^2f(W)
		\rangle_{\mathrm{HS}}
		\right|
		&\leq
		\mathbb{E}\left[
		\|\Xi\|_{\mathrm{HS}}
		\|\nabla^2f(W)\|_{\mathrm{HS}}
		\right]\\
		&\leq
		C\|h\|_{\mathrm{Lip}}
		\mathbb{E}\|\Xi\|_{\mathrm{HS}}.
	\end{align*}
	Moreover, since
	$$
	\left|
		\langle D\delta^\top,A\rangle_{\mathrm{HS}}
		\right|
		\leq
		|D|\,|\delta|\,\|A\|_{\mathrm{op}},
	$$
	we have
	\begin{align*}
		|R|
		&\leq
		C\|h\|_{\mathrm{Lip}}
		\mathbb{E}\left[
		|D|\,|\delta|^2
		\int_0^1
		t\bigl(|\log(t|\delta|)|\lor1\bigr)\,\dif t
		\right]\\
		&\leq
		C\|h\|_{\mathrm{Lip}}
		\mathbb{E}\left[
		|D|\,|\delta|^2
		\bigl(|\log|\delta||\lor1\bigr)
		\right].
	\end{align*}
	
	Therefore, by \eqref{1128gq},
	\begin{align*}
		|\mathbb{E}h(W)-\gamma(h)|
		&=
		\left|
		\mathbb{E}\bigl(
		\Delta f(W)-\langle W,\nabla f(W)\rangle
		\bigr)
		\right|\\
		&\leq
		C\|h\|_{\mathrm{Lip}}
		\left\{
		\mathbb{E}\|\Xi\|_{\mathrm{HS}}
		+
		\frac{1}{\lambda}
		\mathbb{E}\left[
		|D|\,|\delta|^2
		\bigl(|\log|\delta||\lor1\bigr)
		\right]
		\right\}.
	\end{align*}
	Taking the supremum over all $h$ with
	$\|h\|_{\mathrm{Lip}}\leq1$ completes the proof.
\end{proof}

Next, we decompose  $W_n 
 $ in \eqref{LMC-2} into two parts: the martingale part and the remainder term. Then, we use Theorem \ref{L1} to control the distance between the martingale part and $\gamma$.
 
 Define the coordinate map $h_i:\R^d\to\R$ as
    \begin{align}\label{cmap}
        h_i(x)=x_i, \quad {\rm for } \ \ x=(x_1,\ldots,x_d).
    \end{align}
    For each $h_i$, denote by $\varphi_i$ the solution of the following Stein equation
	\begin{align}
		\mathcal{A}\varphi_i=h_i-\pi(h_i),
	\end{align}
	where $\mathcal{A}$ and $\pi$ are the generator  and the unique invariant distribution  of SDE \eqref{LSDE} respectively.  Then we have the following proposition on the regularity of $\varphi_i$.

\begin{proposition}\label{prop1}
	Under \textbf{Assumption}~\ref{Ass1}, let
	$(X_t^x)_{t\geq0}$ be the Langevin dynamics with initial state
	$x$ defined in \eqref{LSDE}, and let $\mathcal A$ and $\pi$
	denote its generator and unique invariant distribution,
	respectively. For any $g\in\mathcal C^3(\mathbb R^d,\mathbb R)$
	satisfying
	$$
		\sup_{x\in\mathbb R^d}
		\|\nabla^k g(x)\|_{\mathrm{op}}
		\leq C,
		\qquad k=1,2,3,
	$$
	the solution of the Stein equation
	$$
		\mathcal Af=g-\pi(g)
	$$
	is given by
	\begin{align}\label{solu}
		f_g(x)
		=
		-\int_0^\infty
		\mathbb E\bigl[g(X_t^x)-\pi(g)\bigr]\,\dif t,
	\end{align}
	and satisfies
	\begin{align}\label{sobo1}
		\|\nabla^k f_g\|_{\mathrm{op}}
		\leq C,
		\qquad k=1,2,3.
	\end{align}
	In particular, if $g=h_i$ is the coordinate map defined in
	\eqref{cmap}, then the solution $\varphi_i$ of
	$$
		\mathcal A\varphi_i=h_i-\pi(h_i)
	$$
	satisfies
	\begin{align}\label{sobo2}
		\|\nabla^k\varphi_i\|_{\mathrm{op}}
		\leq C,
		\qquad k=1,2,3,4.
	\end{align}
\end{proposition}
    For clarity of presentation, we defer the proof of Proposition \ref{prop1} to Appendix~\ref{B}. 
	
According to the definition of $h_i$,   $W_n$   can be represented by 
	\begin{eqnarray}
		 W_n
    &=&
    \frac{\sqrt\eta}{\sqrt n}
    \sum_{k=0}^{n-1}
    \bigl(h(X_k)-\pi(h)\bigr)\nonumber\\
    &=&
    \eta^{1/2}
    \left(
        \frac1{\sqrt n}\sum_{k=0}^{n-1}
        \bigl(h_1(X_k)-\pi(h_1)\bigr),
        \ldots,
        \frac1{\sqrt n}\sum_{k=0}^{n-1}
        \bigl(h_d(X_k)-\pi(h_d)\bigr)
    \right)^\top\nonumber\\
		&=:& (W_{n, 1},\cdots,W_{n, d})^\top ,
	\end{eqnarray}
	where $\pi(h) = (\pi(h_1), \ldots, \pi(h_d))^\top $ with $\pi(h_i) = \int_{\mathbb{R}^d} h_i(x) \, \pi(\dif x)$.
    Then from  \cite[Eq.~3.1]{Lu2022}, we have
	\begin{align}\label{ts1}
		W_n=(\mathcal{H}_{n,1},\cdots,\mathcal{H}_{n,d})^{\top}+(\mathcal{R}_{n,1},\cdots,\mathcal{R}_{n,d})^{\top}:=\mathcal{H}_{n}+\mathcal{R}_{n},
	\end{align}
	where for each $i=1,\cdots,d$,
	\begin{align*}
		\mathcal{H}_{n,i}=-\frac{\sqrt{2}}{\sqrt{n} }\sum_{k=0}^{n-1}\langle \nabla \varphi_i( X _k), \xi_{k+1}\rangle,\quad  \mathcal{R}_{n,i}= -\sum_{j=1}^{6}\mathcal{R}_{n,i,j},
	\end{align*}
	and
\begin{align*}
		\mathcal R_{n,i,1}
=&
\frac{1}{\sqrt{n\eta}}
\bigl[
    \varphi_i(X_0)-\varphi_i(X_n)
\bigr],\\
		\mathcal{R}_{n,i,2}=&\frac{\sqrt{\eta}}{\sqrt{n}}\sum_{k=0}^{n-1}\langle \nabla^2\varphi_i( X _k),\xi_{k+1}\xi_{k+1}^{\top}-I_d\rangle_{\mathrm{HS}},\notag\\
		\mathcal{R}_{n,i,3}=&\frac{\eta}{\sqrt{2n}}\sum_{k=0}^{n-1}\left[\langle\nabla^2\varphi_i( X _k),-\nabla U( X _k)\xi_{k+1}^{\top}\rangle_{\mathrm{HS}}+\langle\nabla^2\varphi_i( X _k),-\xi_{k+1}\nabla U( X _k)^{\top}\rangle_{\mathrm{HS}}\right],\notag\\
		\mathcal{R}_{n,i,4}=&\frac{\sqrt{2}\eta}{\sqrt{n}}\sum_{k=0}^{n-1}\int_{0}^{1}(1-t)^2\langle \nabla^3\varphi_i( X _k+t\Delta X _k),\xi_{k+1}^{\otimes 3}\rangle_{\rm HS}\dif t,\\
		\mathcal{R}_{n,i,5}=&\frac{\eta^{3/2}}{2\sqrt{n}}\sum_{k=0}^{n-1}\langle \nabla^2 \varphi_i( X _k),\nabla U( X _k)\nabla U( X _k)^{\top}\rangle_{\mathrm{HS}}\notag\\
		&-\frac{\eta^{5/2}}{2\sqrt{n}}\sum_{k=0}^{n-1}\int_{0}^{1}(1-t)^2\langle \nabla^3\varphi_i( X _k+t\Delta X _k), \nabla U( X _k)^{\otimes 3}\rangle_{\rm HS}\dif t,\notag\\
			\mathcal{R}_{n,i,6}={}&-\frac{3\eta^{3/2}}{\sqrt{n}}
			\sum_{k=0}^{n-1}\int_{0}^{1}(1-t)^2\Big[
			\nabla^3\varphi_i(X_k+t\Delta X_k)
			[\nabla U(X_k),\xi_{k+1},\xi_{k+1}]\notag\\
			&\qquad\qquad\qquad
			-\sqrt{\frac{\eta}{2}}\,
			\nabla^3\varphi_i(X_k+t\Delta X_k)
			[\nabla U(X_k),\nabla U(X_k),\xi_{k+1}]
			\Big]\dif t,\notag\\
		\Delta  X_k=&-\eta \nabla U( X_k)+\sqrt{2\eta}\xi_{k+1}.\notag
	\end{align*}

     We have the following estimate for the main term $\mathcal H_n$.

\begin{proposition}\label{L21}
			Under the assumptions of Theorem~\ref{T2},
			$$
				\mathcal{W}_1\big(\Sigma^{-1/2}\mathcal{H}_{n}, \gamma\big) \le C\big(\eta^{1/2}d^{3/2}+n^{-1/2}d^{3}\log (\mathrm e\,nd)+(n\eta)^{-1/2}d^2\big),$$
	        where $\Sigma$ is defined in \eqref{cm}.

	\end{proposition}

	\begin{proof}
    The proof of this proposition relies primarily on Theorem \ref{L1}. We apply Theorem~\ref{L1} with
    $$
        W=\Sigma^{-1/2}\mathcal H_n.
    $$

    \textbf{Construction of $D$ and $\delta$}. Let
    $$
        \mathbf X=(X_0,\xi_1,\ldots,\xi_n),
    $$
    and regard $\mathcal H_n$ as a function of $\mathbf X$. Let
    $\xi_1',\ldots,\xi_n'$ be independent copies of
    $\xi_1,\ldots,\xi_n$, independent of $\mathbf X$. For each
    $i\in\{1,\ldots,n\}$, let $\mathbf X^{(i)}$ be obtained from
    $\mathbf X$ by replacing $\xi_i$ with $\xi_i'$, and denote by
    $(X_j^{(i)})_{j\geq0}$ the corresponding perturbed chain. More precisely,
    $$
        X_j^{(i)}=X_j,\qquad 0\leq j\leq i-1,
    $$
    while
    $$
        X_i^{(i)}
        =X_{i-1}-\eta\nabla U(X_{i-1})+\sqrt{2\eta}\,\xi_i'
    $$
    and, for $j\geq i$,
    $$
        X_{j+1}^{(i)}
        =X_j^{(i)}-\eta\nabla U(X_j^{(i)})
        +\sqrt{2\eta}\,\xi_{j+1}.
    $$

    For simplicity, we introduce some additional notations for $i=1,\ldots,n$, 
    \begin{align*}
        t_i
        &:=-\sqrt{2}n^{-1/2}
        \begin{pmatrix}
            \langle\nabla\varphi_1(X_{i-1}),\xi_i\rangle\\
            \vdots\\
            \langle\nabla\varphi_d(X_{i-1}),\xi_i\rangle
        \end{pmatrix},\\
        t_i'
        &:=-\sqrt{2}n^{-1/2}
        \begin{pmatrix}
            \langle\nabla\varphi_1(X_{i-1}),\xi_i'\rangle\\
            \vdots\\
            \langle\nabla\varphi_d(X_{i-1}),\xi_i'\rangle
        \end{pmatrix},\\
        r_i
        &:=\sqrt{2}n^{-1/2}\sum_{j=i}^{n-1}
        \begin{pmatrix}
            \langle\nabla\varphi_1(X_j^{(i)})-\nabla\varphi_1(X_j),\xi_{j+1}\rangle\\
            \vdots\\
            \langle\nabla\varphi_d(X_j^{(i)})-\nabla\varphi_d(X_j),\xi_{j+1}\rangle
        \end{pmatrix},\\
        T_i
        &:=2n^{-1}
        \Bigl(
            \nabla\varphi_a(X_{i-1})^\top
            \nabla\varphi_b(X_{i-1})
        \Bigr)_{1\leq a,b\leq d}.
    \end{align*}

    Let $K$ be uniformly distributed on $\{1,\ldots,n\}$, independently
    of all the above random variables, and set
    $$
        \mathbf X'=\mathbf X^{(K)},
        \qquad
        \mathcal H_n'=\mathcal H_n(\mathbf X').
    $$
    Then $(\mathbf X,\mathbf X')$ is an exchangeable pair and, by construction,
    $$
        \mathcal H_n'-\mathcal H_n=t_K'-t_K-r_K.
    $$
    Define
    \begin{align}\label{D}
        D
        :=\Sigma^{-1/2}(t_K'-t_K)
        =-\sqrt{2}(n\Sigma)^{-1/2}
        \begin{pmatrix}
            \langle\nabla\varphi_1(X_{K-1}),\xi_K'-\xi_K\rangle\\
            \vdots\\
            \langle\nabla\varphi_d(X_{K-1}),\xi_K'-\xi_K\rangle
        \end{pmatrix}.
    \end{align}
    Since $X_{K-1}$ is unchanged by the replacement of $\xi_K$, the
    random vector $D$ is antisymmetric. Moreover, since $K$ is uniform and
    $\E[t_i'\mid\mathbf X]=0$,
    \begin{align*}
        \E[D\mid\mathbf X]
        &=-\frac1n\Sigma^{-1/2}\sum_{i=1}^n t_i
        =-\frac1n\Sigma^{-1/2}\mathcal H_n
        =-\frac1n W.
    \end{align*}
    Thus, the first condition in \eqref{ASSU} holds with $\lambda=n^{-1}$.
    Furthermore,
    \begin{align}\label{delta}
        \delta
        :=\Sigma^{-1/2}(\mathcal H_n'-\mathcal H_n)
        =D-\Sigma^{-1/2}r_K.
    \end{align}

    \textbf{Construction of $\Xi$ and its estimate}. For each
    $i=1,\ldots,n$, define
    $$
        Q_i
        :=-\E\left[(t_i'-t_i)r_i^\top\mid\mathbf X\right],
    $$
    where the conditional expectation is taken with respect to $\xi_i'$.
    A straightforward calculation gives
    \begin{align*}
        \E(D\delta^\top\mid\mathbf X)
        &=\frac{\Sigma^{-1/2}}{n}
        \sum_{i=1}^n\left(T_i+t_it_i^\top+Q_i\right)\Sigma^{-1/2}\\
        &=\frac{\Sigma^{-1/2}}{n}
        \left[2\Sigma+
        \left(\sum_{i=1}^nT_i-\Sigma\right)
        +\left(\sum_{i=1}^nt_it_i^\top-\Sigma\right)
        +\sum_{i=1}^nQ_i\right]\Sigma^{-1/2}\\
        &=\frac2n\left[I_d+\Sigma^{-1/2}
        \left(
            \frac{\sum_{i=1}^nT_i-\Sigma}{2}
            +\frac{\sum_{i=1}^nt_it_i^\top-\Sigma}{2}
            +\frac{\sum_{i=1}^nQ_i}{2}
        \right)\Sigma^{-1/2}\right].
    \end{align*}
    Thus,
	    \begin{align}\label{RS}
        \Xi
        =\frac{\Sigma^{-1/2}}2
        \left[
            \left(\sum_{i=1}^nT_i-\Sigma\right)
            +\left(\sum_{i=1}^nt_it_i^\top-\Sigma\right)
            +\sum_{i=1}^nQ_i
	        \right]\Sigma^{-1/2}.
	    \end{align}
	    The boundedness of $\nabla\varphi$, the contraction of the
	    perturbed chains, and the Gaussian moment bounds imply that
	    $W,D,\delta$, and $\Xi$ satisfy all the integrability
	    requirements in Assumption~\ref{1544gl}.

    To estimate $\E\|\Xi\|_{\rm HS}$, define
	    \begin{align}
        R_1
        &=\frac12\E\left\|\sum_{i=1}^nT_i-\Sigma\right\|_{\mathrm{HS}},\quad
        R_2
        =\frac12\E\left\|\sum_{i=1}^nt_it_i^\top-\Sigma\right\|_{\mathrm{HS}},\notag\\
        R_3
        &=\frac12\E\left\|\sum_{i=1}^nQ_i\right\|_{\mathrm{HS}}.
    \end{align}

	 To analyze $R_1$, we need the following two lemmas.
		\begin{lemma}\label{M}
		For any matrix $Q$ and positive definite symmetric matrices $P, R$, the following inequalities hold:
		\begin{align*}
			\|PQ\|_{\mathrm{HS}} &\leq \lambda_{\max}(P) \|Q\|_{\mathrm{HS}}, \\
			\|R^{-1} Q R\|_{\mathrm{HS}} &\leq \frac{\lambda_{\max}(R)}{\lambda_{\min}(R)} \|Q\|_{\mathrm{HS}},
		\end{align*}
		where $\lambda_{\max}(\cdot)$ and $\lambda_{\min}(\cdot)$ denote the largest and smallest eigenvalues, respectively.
	\end{lemma}
\begin{proof}
The inequalities follow from the submultiplicativity of the Hilbert--Schmidt norm,
			$\|AB\|_{\mathrm{HS}} \le \|A\|_{\rm{op}}\|B\|_{\mathrm{HS}}$,
			and the identities $\|P\|_{\rm{op}}=\lambda_{\max}(P)$,
			$\|R^{-1}\|_{\rm{op}}=1/\lambda_{\min}(R)$ for symmetric positive definite matrices;
			see, e.g., \cite[Chapter~5]{HornJohnson}.
\end{proof}

		\begin{lemma}\label{L23}
			Let $X_0\sim \pi_\eta$. Then for any $1 \leq i,j \leq d$  and
	             $\eta\in (0,\alpha/(2\beta^2)) $, it holds that
			$$
			\mathbb{E} \left[2 n^{-1} \sum_{I=1}^n \nabla \varphi_i( X _{I-1})^\top    \nabla \varphi_j( X _{I-1}) - \pi \left( 2\nabla \varphi_i^\top    \nabla \varphi_j \right) \right]^2 \leq C \big(\eta d+\sqrt{d}(n\eta)^{-1}\big).
			$$
		\end{lemma}
		For the sake of reading, we put the proof of Lemma \ref{L23} in subsection \ref{sub33}.

		According to  Lemma \ref{L23} and the Cauchy--Schwarz inequality, we have
		\begin{align}
			R_1
			\leq& \sqrt{ \sum_{i,j=1}^d \mathbb{E} \left[ 2n^{-1} \sum_{I=1}^n \nabla \varphi_i( X _{I-1})^\top   \nabla \varphi_j( X _{I-1}) - \pi \left( 2\nabla \varphi_i^\top   \nabla \varphi_j \right) \right]^2 } \notag \\
			\leq& C \big(\eta^{1/2} d^{3/2}+d^{5/4}(n\eta)^{-1/2}\big).
		\end{align}
		
		Next, note that
		$$
		R_2 \leq R_1+\frac12\mathbb{E}\left\|\sum_{I=1}^n(t_It_I^\top-T_I)\right\|_{\mathrm{HS}}=:R_1+\widetilde R_2.
		$$
       Let
$$
    G_k:=
    \begin{pmatrix}
       \nabla\varphi_1(X_k)^\top\\
       \vdots\\
       \nabla\varphi_d(X_k)^\top
    \end{pmatrix},
    \qquad
    Y_I:=\frac{2}{n}G_{I-1}
         (\xi_I\xi_I^\top-I_d)G_{I-1}^\top.
$$
For $\mathcal F_I:=\sigma(X_0,\xi_1,\ldots,\xi_I)$, the matrix $G_{I-1}$ is $\mathcal F_{I-1}$-measurable and
$$
    \E[Y_I\mid\mathcal F_{I-1}]=0.
$$
Consequently, the matrices $Y_I$ are orthogonal in $L^2$ with respect to the Hilbert--Schmidt inner product, and
\begin{align*}
    \widetilde R_2
    &\leq
    \frac12
    \left(
        \mathbb E
        \left\|
            \sum_{I=1}^nY_I
        \right\|_{\mathrm{HS}}^2
    \right)^{1/2}=
    \frac12
    \left(
        \sum_{I=1}^n
        \mathbb E\|Y_I\|_{\mathrm{HS}}^2
    \right)^{1/2}\leq
    Cn^{-1/2}d.
\end{align*}
 Here we used
$
    \sup_{k\geq0}\|G_k\|_{\mathrm{op}}\leq C.
$
Indeed, for any $u\in\mathbb R^d$ with $|u|=1$, let
$
h_u(x)=\langle u,x\rangle
$
and
$
\varphi_u=\sum_{i=1}^d u_i\varphi_i.
$
Then
$
    \mathcal A\varphi_u=h_u-\pi(h_u),
$
and Proposition~\ref{prop1} gives
$
    |G_k^\top u|
    =
    |\nabla\varphi_u(X_k)|
    \leq C.
$ Taking the supremum over $|u|=1$ yields the desired bound.  We also used
$
    \E\|\xi_I\xi_I^\top-I_d\|_{\mathrm{HS}}^2\leq Cd^2.
$

Since $d\geq1$ and $\eta\leq\alpha/(2\beta^2)$, this error term $\widetilde R_2$ is
absorbed into $Cd^{5/4}(n\eta)^{-1/2}$.
        Consequently,
		\begin{align}\label{R_2}
			R_2 \leq C \big(\eta^{1/2} d^{3/2}+d^{5/4}(n\eta)^{-1/2}\big).
		\end{align}
		
				It remains to estimate $R_3$. By conditional Jensen's inequality,
		$$
			R_3
			\leq
			\frac12\mathbb E
			\left\|
				\sum_{I=1}^n r_I(t_I'-t_I)^\top
			\right\|_{\mathrm{HS}}.
		$$
		Denote by $N_{a,b}^I$ the $(a,b)$-th entry of
		$r_I(t_I'-t_I)^\top$. Then
		\begin{align}\label{08072}
			\mathbb E
			\left\|
				\sum_{I=1}^n r_I(t_I'-t_I)^\top
			\right\|_{\mathrm{HS}}
			&=
			\mathbb E
			\left[
				\sum_{a,b=1}^d
				\left(
					\sum_{I=1}^n N_{a,b}^I
				\right)^2
			\right]^{1/2}\leq
			\left[
				\sum_{a,b=1}^d
				\mathbb E
				\left(
					\sum_{I=1}^n N_{a,b}^I
				\right)^2
			\right]^{1/2}.
		\end{align}
		
		We estimate the second moment in \eqref{08072}. Set
		$$
			q_\eta
			:=
			1-2\alpha\eta+\beta^2\eta^2,
			\qquad
			\rho_\eta:=q_\eta^{1/2}.
		$$
		Under the step-size condition in Theorem~\ref{T2},
		$q_\eta\in(0,1)$, and
		\begin{equation}\label{eq:rho-eta}
			1-\rho_\eta
			=
			\frac{1-q_\eta}{1+\rho_\eta}
			\geq c\eta.
		\end{equation}
		For $1\leq I\leq j$, write
		$$
			A_{I,j}^{(a)}
			:=
			\nabla\varphi_a(X_j^{(I)})
			-
			\nabla\varphi_a(X_j)
		$$
		and
		$$
			B_I^{(b)}
			:=
			\left\langle
				\nabla\varphi_b(X_{I-1}),
				\xi_I'-\xi_I
			\right\rangle.
		$$
		Up to an irrelevant sign,
		$$
			\sum_{I=1}^n N_{a,b}^I
			=
			\frac{2}{n}
			\sum_{j=1}^{n-1}
			\sum_{I=1}^j
			\left\langle
				A_{I,j}^{(a)},\xi_{j+1}
			\right\rangle
			B_I^{(b)}.
		$$
		
		Let
		$$
			\mathcal F_j'
			:=
			\sigma
			\left(
				X_0,\xi_1,\ldots,\xi_j,
				\xi_1',\ldots,\xi_n'
			\right).
		$$
		The coefficients $A_{I,j}^{(a)}$ and $B_I^{(b)}$,
		$I\leq j$, are $\mathcal F_j'$-measurable, whereas
		$\xi_{j+1}$ is independent of $\mathcal F_j'$ and centered.
		Consequently, the summands indexed by $j$ form martingale
		differences. Hence,
		\begin{align*}
			\mathbb E
			\left(
				\sum_{I=1}^n N_{a,b}^I
			\right)^2
			&=
			\frac{4}{n^2}
			\sum_{j=1}^{n-1}
			\mathbb E
			\left(
				\sum_{I=1}^j
				\left\langle
					A_{I,j}^{(a)},\xi_{j+1}
				\right\rangle
				B_I^{(b)}
			\right)^2\\
			&=
			\frac{4}{n^2}
			\sum_{j=1}^{n-1}
			\mathbb E
			\left|
				\sum_{I=1}^j
				B_I^{(b)}A_{I,j}^{(a)}
			\right|^2\\
			&\leq
			\frac{8}{n^2}
			\sum_{j=1}^{n-1}
			\sum_{I=1}^j
			\sum_{J=I}^j
			\mathbb E
			\left[
				|B_I^{(b)}|\,|B_J^{(b)}|
				|A_{I,j}^{(a)}|\,|A_{J,j}^{(a)}|
			\right].
		\end{align*}
		
		It follows from \eqref{LMC-1} and
		\textbf{Assumption}~\ref{Ass1} that
		\begin{align*}
    |X_j^{(I)}-X_j|^2
    &\leq q_\eta^{\,j-I}|X_I^{(I)}-X_I|^2=2\eta q_\eta^{\,j-I}|\xi_I'-\xi_I|^2,
    \qquad j\geq I.
\end{align*}
		Therefore, by \eqref{sobo2}, for $j\geq I$,
		\begin{align}\label{GJ}
			|A_{I,j}^{(a)}|
			\leq
			C\sqrt{\eta}\,
			\rho_\eta^{\,j-I}
			|\xi_I'-\xi_I|.
		\end{align}
		Moreover,
		$$
			|B_I^{(b)}|
			\leq
			C|\xi_I'-\xi_I|.
		$$
		Since
		$
			\mathbb E|\xi_I'-\xi_I|^4\leq Cd^2,
		$
		Cauchy--Schwarz's inequality gives, uniformly in $I,J$,
		$$
			\mathbb E
			\left[
				|\xi_I'-\xi_I|^2
				|\xi_J'-\xi_J|^2
			\right]
			\leq Cd^2.
		$$
		Consequently,
		\begin{align}
			\mathbb E
			\left(
				\sum_{I=1}^n N_{a,b}^I
			\right)^2
			&\leq
			Cn^{-2}d^2
			\sum_{j=1}^{n-1}
			\sum_{I=1}^j
			\sum_{J=I}^j
			\eta
			\rho_\eta^{\,2j-I-J}\notag\\
			&\leq
			Cn^{-2}d^2
			\sum_{j=1}^{n-1}
			\eta
			\left(
				\sum_{m=0}^{j-1}\rho_\eta^m
			\right)^2\notag\\
			&\leq
			Cn^{-2}d^2
			\sum_{j=1}^{n-1}
			\frac{\eta}{(1-\rho_\eta)^2}\notag\\
			&\leq
			C\frac{d^2}{n\eta},
		\end{align}
		where the last inequality follows from \eqref{eq:rho-eta}.
		The same estimate holds for every $1\leq a,b\leq d$.
		Thus, by \eqref{08072},
		\begin{align}\label{R_3}
			R_3
			\leq
			C(n\eta)^{-1/2}d^2.
		\end{align}
		
		Combining \eqref{RS}--\eqref{R_2} and \eqref{R_3}, we obtain
		\begin{align}\label{E_2}
			\mathbb E\|\Xi\|_{\mathrm{HS}}
			\leq
			C\left(
				\eta^{1/2}d^{3/2}
				+
				d^2(n\eta)^{-1/2}
			\right).
		\end{align}
		
		\textbf{Control of
		$\mathbb E[|D||\delta|^2]$ and the logarithmic term.}
		From \eqref{D}, we have, for every $r\geq2$,
		\begin{align}
			\|D\|_r
			&\leq
			Cn^{-1/2}\sqrt d\,
			\|\xi_K'-\xi_K\|_r
			\leq
			C_rn^{-1/2}d.
		\end{align}
		Set
		$$
			\Psi_i
			=
			\sum_{j=K}^{n-1}
			\left\langle
				A_{K,j}^{(i)},\xi_{j+1}
			\right\rangle,
			\qquad 1\leq i\leq d.
		$$
		Conditionally on $K,\xi_K'$, the summands in $\Psi_i$
		form martingale differences. Hence, for $2\leq r\leq4$,
		the Burkholder inequality and \eqref{GJ} imply
		\begin{align*}
			\|\Psi_i\|_r
			&\leq
			C_r
			\left\|
				\left(
					\sum_{j=K}^{n-1}
					|A_{K,j}^{(i)}|^2
				\right)^{1/2}
			\right\|_r\leq
			C_r
			\left\|
				\left(
					\eta
					\sum_{m=0}^{\infty}
					q_\eta^m
				\right)^{1/2}
				|\xi_K'-\xi_K|
			\right\|_r\leq
			C_r\sqrt d.
		\end{align*}
		Therefore, by Minkowski's inequality in $L^{r/2}$,
		\begin{align*}
			\left\|
				\left(
					\sum_{i=1}^d\Psi_i^2
				\right)^{1/2}
			\right\|_r^2
			&=
			\left\|
				\sum_{i=1}^d\Psi_i^2
			\right\|_{r/2}\leq
			\sum_{i=1}^d
			\|\Psi_i\|_r^2
			\leq
			C_rd^2.
		\end{align*}
		In view of \eqref{delta}, it follows that, for
		$2\leq r\leq4$,
		\begin{align}\label{eq:D-delta-moments}
			\|D\|_r+\|\delta\|_r
			\leq
			C_rn^{-1/2}d.
		\end{align}
		By H\"older's inequality and \eqref{eq:D-delta-moments},
		\begin{align}
			\mathbb E\bigl[|D||\delta|^2\bigr]
			&\leq
			\|D\|_3\|\delta\|_3^2
			\leq
			Cn^{-3/2}d^3.
		\end{align}
		
		It remains to estimate the logarithmic factor. Let
		$$
			L:=\log(\mathrm e\,nd),
			\qquad
			a:=L^{-1}\in(0,1].
		$$
		For every $x>0$,
		\begin{equation}
			|\log x|\vee1
			\leq
			\frac{C}{a}
			\left(x^a+x^{-a}\right).
		\end{equation}
		Hence,
		\begin{align*}
			&\mathbb E
			\left[
				|D||\delta|^2
				\bigl(|\log|\delta||\vee1\bigr)
			\right]\leq
			CL
			\left\{
				\mathbb E\bigl[|D||\delta|^{2+a}\bigr]
				+
				\mathbb E\bigl[|D||\delta|^{2-a}\bigr]
			\right\}.
		\end{align*}
		For $s\in[1,3]$, H\"older's inequality and
		\eqref{eq:D-delta-moments} yield
		$$
			\mathbb E\bigl[|D||\delta|^s\bigr]
			\leq
			\|D\|_{s+1}\|\delta\|_{s+1}^s
			\leq
			Cn^{-(s+1)/2}d^{s+1}.
		$$
		In particular,
		$$
			n^{-(3\pm a)/2}d^{3\pm a}
			=
			n^{-3/2}d^3
			\left(\frac{d}{\sqrt n}\right)^{\pm a}
			\leq
			Cn^{-3/2}d^3,
		$$
		because $a=1/\log(\mathrm e\,nd)$. We therefore conclude that
		\begin{align}\label{ln2}
			\mathbb E
			\left[
				|D||\delta|^2
				\bigl(|\log|\delta||\vee1\bigr)
			\right]
			\leq
			Cn^{-3/2}d^3\log(\mathrm e\,nd).
		\end{align}
		Combining \eqref{E_2}, \eqref{ln2}, and
		Theorem~\ref{L1} completes the proof.
	\end{proof}

\subsection{Estimate  the Remainder}\label{sub32}

Recalling \eqref{ts1}, we write $W_n$ as the sum of the martingale part $\mathcal H_n$ and the remainder $\mathcal R_n$. We have already obtained the distance bound between $\Sigma^{-1/2}\mathcal H_n$ and $\gamma$. Next, we estimate the remainder term $\mathcal R_n$.
\begin{proposition}\label{L22}
			Under the assumptions of Theorem~\ref{T2}, the remainder $\mathcal R_n$ in \eqref{ts1} satisfies
		$$
		\mathbb{E}\big| \Sigma^{-1/2} \mathcal{R}_{n} \big| \leq C \big((n\eta)^{-1/2}d+\eta^{1/2} d^{2}+\eta^{3/2}n^{1/2}d^{5/2}\big).
		$$
	\end{proposition}
	
	To prove  Proposition \ref{L22}, we need the following  two crucial lemmas whose proofs are postponed to Subsection \ref{sub33}.
		\begin{lemma}\label{Bound}
			Let $V(x)=|x|^{4}+1$ for $x\in \mathbb{R}^d$. For any $\eta\in(0,\frac{\alpha}{2\beta^2})$, there exists a positive constant $C$ depending only on $\alpha$ and $\beta$ such that
		\begin{align*}
			\pi_{\eta}(V)\leq C d^2.
		\end{align*}
	\end{lemma}
	
	\begin{lemma}\label{873}
For $1\leq i\leq d$, we have
\begin{align}
    &\mathbb{E}\left|
        \sum_{k=0}^{n-1}
        \left\langle
            \nabla^2\varphi_i(X_k),
            \xi_{k+1}\xi_{k+1}^{\top}-I_d
        \right\rangle_{\mathrm{HS}}
    \right|^2
    \leq Cnd,
    \label{461}\\
    &\mathbb{E}\left|
        \sum_{k=0}^{n-1}
        \left\langle
            \nabla^2\varphi_i(X_k),
            -\nabla U(X_k)\xi_{k+1}^{\top}
        \right\rangle_{\mathrm{HS}}
    \right|^2
    \leq Cnd.
    \label{1547}
\end{align}
Moreover, for any $t\in(0,1)$,
\begin{align}\label{1553}
    \mathbb{E}\left|
        \sum_{k=0}^{n-1}
        \left\langle
            \nabla^3\varphi_i(X_k+t\Delta X_k),
            \xi_{k+1}^{\otimes3}
        \right\rangle_{\mathrm{HS}}
    \right|^2
    \leq C\bigl(nd^3+\eta n^2d^4\bigr).
\end{align}
\end{lemma}
	Now, with the above two lemmas at hand,  we  can prove  Proposition \ref{L22}.
	
	\begin{proof}[Proof of Proposition \ref{L22}]
		Denote $R_{j,n} = (\mathcal{R}_{n,1,j}, \cdots, \mathcal{R}_{n,d,j})^\top $. Clearly, we have
		$$
		\mathcal R_n  = -\sum_{j=1}^6 R_{j,n}, \quad \text{and} \quad \mathbb{E}|\mathcal R_n | \leq \sum_{j=1}^6 \mathbb{E}|R_{j,n}|.
		$$
		Since $X_0\sim \pi_\eta$, $\{X_n\}_{n\geq0}$ is a stationary Markov chain, which means that $X_n\sim\pi_\eta$ for any $n\geq0$. It is straightforward to see that \textbf{Assumption} \ref{Ass1} implies
		\begin{align}
			|\nabla U(x)|^2 &\leq C_1 |x|^2 , \label{AB1} \\
			\langle x, -\nabla U(x) \rangle &\leq -C_2 |x|^2,
		\end{align}
		for all $x \in \mathbb{R}^d$, where the positive constants $C_1$ and $C_2$ depend on $\alpha, \beta$. We next estimate the six vectors $R_{j,n}$ successively.

		\textit{Term $R_{1,n}$}.  It follows from the boundedness of $\|\nabla \varphi\|_{\mathrm{op}}$ and Lemma \ref{Bound} that
		\begin{align*}
		\mathbb{E}|R_{1,n}| &\leq C (n \eta)^{-1/2} d^{1/2} \mathbb{E}| X _0 -  X _n| \leq C (n \eta)^{-1/2} d^{1/2} [\pi_{\eta}(V)]^{1/4} \leq C (n \eta)^{-1/2} d.
		\end{align*}

	\textit{Term $R_{2,n}$.}
By \eqref{461},
\begin{align*}
    \mathbb E|R_{2,n}|
    &\leq
    \left\{
        \frac{\eta}{n}
        \sum_{i=1}^d
        \mathbb E
        \left|
            \sum_{k=0}^{n-1}
            \left\langle
                \nabla^2\varphi_i(X_k),
                \xi_{k+1}\xi_{k+1}^\top-I_d
            \right\rangle_{\mathrm{HS}}
        \right|^2
    \right\}^{1/2}\leq C\sqrt\eta\,d.
\end{align*}

     \textit{Term $R_{3,n}$.}
By \eqref{1547},
\begin{align*}
    \mathbb E|R_{3,n}|
    &\leq
    C\left\{
        \frac{\eta^2}{n}
        \sum_{i=1}^d
        \mathbb E
        \left|
            \sum_{k=0}^{n-1}
            \left\langle
                \nabla^2\varphi_i(X_k),
                -\nabla U(X_k)\xi_{k+1}^\top
            \right\rangle_{\mathrm{HS}}
        \right|^2
    \right\}^{1/2}\leq C\eta d.
\end{align*}
        
		\textit{Term $R_{4,n}$}. By \eqref{1553},
          \begin{align*}
    \mathbb{E}|R_{4,n}|
    \leq&\sqrt{\frac{2\eta^2}{n}\sum_{i=1}^d\E\left|\sum_{k=0}^{n-1}\int_0^1(1-t)^2\langle \nabla^3\varphi_i( X _k+t\Delta X _k),\xi_{k+1}^{\otimes 3}\rangle_{\rm HS}\dif t\right|^2}\\
        \leq& C(\eta d^2+\eta^{3/2}n^{1/2}d^{5/2}).
         \end{align*}

       \textit{Term $R_{5,n}$}. From Lemma \ref{Bound} and the boundedness of $\|\nabla^2 \varphi_i\|_{\rm{op}}$ and $\|\nabla^3 \varphi_i\|_{\rm{op}}$,
\begin{align*}
\mathbb{E}\lvert R_{5,n}\rvert
&\leq \mathbb{E}\sqrt{
    \sum_{i=1}^d
    \bigg[
        \sum_{k=0}^{n-1}
        \Big(
            \frac{\eta^{3/2}}{2 n^{1/2}}
            \,\|\nabla^2 \varphi_i\|_{\mathrm{op}}
            \,\|\nabla U(X_k)\|^2
            +
            \frac{\eta^{5/2}}{2 n^{1/2}}
            \,\|\nabla^3 \varphi_i\|_{\mathrm{op}}
            \,\|\nabla U(X_k)\|^3
        \Big)
    \bigg]^2
}\\
&\leq
\mathbb{E}\sqrt{
    \sum_{i=1}^d
    \bigg[
        \sum_{k=0}^{n-1}
        \frac{\eta^{3/2}}{2 n^{1/2}}
        \,\|\nabla^2 \varphi_i\|_{\mathrm{op}}
        \,\|\nabla U(X_k)\|^2
    \bigg]^2
}
\\ &\quad+
\mathbb{E}\sqrt{
    \sum_{i=1}^d
    \bigg[
        \sum_{k=0}^{n-1}
        \frac{\eta^{5/2}}{2 n^{1/2}}
        \,\|\nabla^3 \varphi_i\|_{\mathrm{op}}
        \,\|\nabla U(X_k)\|^3
    \bigg]^2
}
\\
&\leq
\sqrt{\mathbb{E}
    \sum_{i=1}^d
    \bigg[
        \sum_{k=0}^{n-1}
        \frac{\eta^{3/2}}{2 n^{1/2}}
        \,\|\nabla^2 \varphi_i\|_{\mathrm{op}}
        \,\|\nabla U(X_k)\|^2
    \bigg]^2
}
\\ &\quad+
\mathbb{E}\left(
    \sum_{i=1}^d
    \bigg[
        \sum_{k=0}^{n-1}
        \frac{\eta^{5/2}}{2 n^{1/2}}
        \,\|\nabla^3 \varphi_i\|_{\mathrm{op}}
        \,\|\nabla U(X_k)\|^3
    \bigg]\right)
\\
&\leq C\bigl(\eta^{3/2} n^{1/2} d^{3/2}
        + \eta^{5/2} n^{1/2} d^{5/2}\bigr).
\end{align*}
Since $d\geq1$ and $\eta\leq\alpha/(2\beta^2)$, both terms are
absorbed into $C\eta^{3/2}n^{1/2}d^{5/2}$, after adjusting
$C$.

        \textit{Term $R_{6,n}$}. As in the estimate of \textit{Term $R_{5,n}$},
\begin{align*}
\E|R_{6,n}|
&\leq
 \frac{C\eta^{3/2}}{\sqrt n}
 \left\{
 \E\sum_{i=1}^d
 \left(\sum_{k=0}^{n-1}
       |\nabla U(X_k)|\,|\xi_{k+1}|^2
 \right)^2
 \right\}^{1/2}\\
&\quad+
 \frac{C\eta^{2}}{\sqrt n}
 \left\{
 \E\sum_{i=1}^d
 \left(\sum_{k=0}^{n-1}
       |\nabla U(X_k)|^2\,|\xi_{k+1}|
 \right)^2
 \right\}^{1/2}\\
&\leq C\left(\eta^{3/2}n^{1/2}d^2
             +\eta^2n^{1/2}d^2\right)
 \leq C\eta^{3/2}n^{1/2}d^2.
\end{align*}
In the last inequality we again used the step-size bound, with its
fixed factor absorbed into $C$.

		Combining the inequality $\mathbb{E}\big| \Sigma^{-1/2} \mathcal{R}_{n} \big|\leq \|\Sigma^{-1/2}\|_{\rm op}\E|\mathcal R_{n}|$ with the above six estimates gives the desired result.
	\end{proof}

\subsection{Proof of Theorem \ref{T2}}\label{sub33}

Before proving Theorem \ref{T2}, we first prove 
 Lemmas \ref{L23}, \ref{Bound} and  \ref{873}. 

\begin{proof}[Proof of Lemma \ref{L23}]
		We divide the proof into two steps. The first step is to bound
		\begin{align}
			\left|\pi_\eta (2\nabla \varphi_i^\top   \nabla \varphi_j) - \pi(2\nabla \varphi_i^\top   \nabla \varphi_j)\right|,
		\end{align}
		and the second step is to bound
		\begin{align}\label{L232}
			\mathbb{E}\left[2n^{-1}\sum_{I=1}^n \nabla \varphi_i( X _{I-1})^\top    \nabla \varphi_j( X _{I-1}) - \pi_\eta(2\nabla \varphi_i^\top   \nabla \varphi_j)\right]^2.
		\end{align}
		
		\textbf{Step 1:} Let
		$$
		l_{ij}(x):=2\nabla\varphi_i(x)^\top\nabla\varphi_j(x).
		$$
        By \eqref{sobo2},
        \begin{align*}
        |\nabla l_{ij}(x)|
        &\leq 2\|\nabla^2\varphi_i(x)\|_{\mathrm{op}}|\nabla\varphi_j(x)|
        +2\|\nabla^2\varphi_j(x)\|_{\mathrm{op}}|\nabla\varphi_i(x)|
        \leq C.
        \end{align*}
       Hence, $l_{ij}$ is Lipschitz with a Lipschitz constant independent
of $i,j$. Moreover, by \cite[Theorem~12]{durmus2019analysis},
$
    \mathcal W_2^2(\pi_\eta,\pi)
    \leq
    \frac{2\beta}{\alpha}\eta d.
$
Consequently,
\begin{align}\label{11}
    \left|\pi_\eta(l_{ij})-\pi(l_{ij})\right|
    &\leq
    \|l_{ij}\|_{\mathrm{Lip}}
    \mathcal W_1(\pi_\eta,\pi)\notag\\
    &\leq
    C\mathcal W_2(\pi_\eta,\pi)
    \leq
    C\eta^{1/2}d^{1/2}.
\end{align}
        
		\textbf{Step 2:} Denote
		$$
		Z_I = 2\nabla \varphi_i( X _{I-1})^\top    \nabla \varphi_j( X _{I-1}), \quad Z = \pi_\eta (2\nabla \varphi_i^\top   \nabla \varphi_j).
		$$
		Note that
		\begin{align*}
			& \mathbb{E} \left[ n^{-1} \sum_{I=1}^n Z_I - Z \right]^2 \\
			=\; & n^{-2} \sum_{I=1}^n \mathbb{E} (Z_I - Z)^2 + 2 n^{-2} \sum_{1 \leq I < J \leq n} \mathbb{E} \left[ (Z_I - Z)(Z_J - Z) \right] \\
            =&n^{-2} \sum_{I=1}^n \mathbb{E} (Z_I - Z)^2+ 2 n^{-2} \sum_{1 \leq I < J \leq n}\operatorname{Cov}(Z_I, Z_J)\\
			:= & S_1 + S_2.
		\end{align*}
		
		For the first term, the boundedness of $\|\nabla \varphi\|_{\mathrm{op}}$ implies that for any $1\leq I\leq n$,
		$$
		\mathbb{E} (Z_I - Z)^2 \leq 2 \mathbb{E} Z_I^2 + 2 Z^2 \leq C,
		$$
		hence $S_1 \leq C n^{-1}$.

		To handle  $S_2$, set
		$$
		l(x) := 2\nabla \varphi_i(x)^\top    \nabla \varphi_j(x), \quad p_k l(x) := \mathbb{E}[l( X _k) \mid  X _0 = x].
		$$
		By stationarity, we get
		\begin{align*}
			\operatorname{Cov}(Z_I, Z_J) &=  \operatorname{Cov}(l( X _0), l( X _{J-I})) \\
			&= \mathbb{E}[l( X _0) l( X _{J-I})] - \mathbb{E} l( X _0) \mathbb{E} l( X _{J-I}) \\
			&= \mathbb{E}[l( X _0) (p_{J-I} l( X _0) - \mathbb{E} l( X _0))]\\
        &= \mathbb{E}[l( X _0) (p_{J-I} l( X _0) - \mathbb{E} p_{J-I} l( X _0))].
		\end{align*}
		Let $k=J-I$. Since $l$ is bounded,
		we have
		\begin{align*}
			\operatorname{Cov}(Z_I, Z_J) &\leq C\mathbb{E} | p_k l( X _0) - \mathbb{E} p_k l( X _0) | \\
			&=C \int_{\mathbb{R}^d} \left| \int_{\mathbb{R}^d} (p_k l(x) - p_k l(y)) \pi_\eta(\mathrm{d}y) \right| \pi_\eta(\mathrm{d}x) \\
			&\leq C\int_{\mathbb{R}^d} \int_{\mathbb{R}^d} |p_k l(x) - p_k l(y)| \pi_\eta(\mathrm{d}y) \pi_\eta(\mathrm{d}x).
		\end{align*}
		
		By \textbf{Assumption} \ref{Ass1}, for any $x,y$,
		\begin{align*}
			|p_k l(x) - p_k l(y)| &= \left| \mathbb{E}[l( X _k) \mid  X _0 = x] - \mathbb{E}[l( X _k) \mid  X _0 = y] \right| \\
			&= \left| \mathbb{E}[l( X _k') - l( X _k) \mid  X _0 = x,  X _0' = y] \right| \\
			&\leq C \mathbb{E} \left[ | X _k' -  X _k| \mid  X _0 = x,  X _0' = y \right] \\
			&\leq C (1 - 2\eta \alpha + \eta^2 \beta^2)^{k/2} |y - x|.
		\end{align*}
		
		Therefore, by Lemma~\ref{Bound},
		\begin{align*}
			\operatorname{Cov}(Z_I, Z_J) &\leq C (1 - 2\eta \alpha + \eta^2 \beta^2)^{(J-I)/2} \mathbb{E}| X _0 -  X _0'| \\
			&\leq C \sqrt{d} (1 - 2\eta \alpha + \eta^2 \beta^2)^{(J-I)/2}.
		\end{align*}
		
			Under the step-size condition $\eta\in(0,\alpha/(2\beta^2))$,
			$1-\sqrt{1-2\alpha\eta+\beta^2\eta^2}\geq c\eta$.
			Therefore,
		\begin{align*}
			S_2 &\leq C n^{-2} \sqrt{d} \sum_{1 \leq I < J \leq n} (1 - 2\eta \alpha + \eta^2 \beta^2)^{(J-I)/2} \\
			&\leq C n^{-2} \sqrt{d} \sum_{I=0}^{n-1} \sum_{K=1}^{n-1-I} (1 -2 \eta \alpha + \eta^2 \beta^2)^{K/2} \\
			&\leq C \frac{\sqrt{d}}{n \eta}.
		\end{align*}
		
		Combining \eqref{11} with the above estimate of \eqref{L232}, we obtain
        $$
        \mathbb E\left[n^{-1}\sum_{I=1}^n Z_I-\pi(l_{ij})\right]^2
        \leq C\left(\eta d+\frac{\sqrt d}{n\eta}\right),
        $$
        which concludes the proof.
	\end{proof}
\begin{proof}[Proof of Lemma \ref{Bound}]
Set
$
F_\eta(x)=x-\eta\nabla U(x)
$
and
$
q_\eta=1-2\alpha\eta+\beta^2\eta^2.
$
By $\alpha$-strong convexity and $\beta$-smoothness, for all
$x,y\in\mathbb R^d$,
\begin{equation}\label{eq:F-contraction-moment}
    |F_\eta(x)-F_\eta(y)|^2
    \leq
    q_\eta|x-y|^2.
\end{equation}
Since $0<\eta<\alpha/(2\beta^2)$,
\begin{equation}\label{eq:q-gap-moment}
    1-q_\eta
    =
    \eta(2\alpha-\beta^2\eta)
    \geq
    \frac{3\alpha}{2}\eta.
\end{equation}

Let $(\widetilde X_k)_{k\geq0}$ be the LMC started from
$\widetilde X_0=0$. By
\cite[Theorem~11(iii)]{durmus2019high},
the law of $\widetilde X_k$ converges to the unique invariant
distribution $\pi_\eta$ in total variation as $k\to\infty$.
Consequently, for every bounded measurable function
$f:\mathbb R^d\to\mathbb R$,
\begin{equation}\label{2037}
    \lim_{k\to\infty}\E f(\widetilde X_k)
    =
    \pi_\eta(f).
\end{equation}

Write
$
m_{2,k}=\mathbb E|\widetilde X_k|^2
$
and
$
m_{4,k}=\mathbb E|\widetilde X_k|^4.
$
Since $\nabla U(0)=0$, \eqref{eq:F-contraction-moment} implies
$
|F_\eta(x)|^2\leq q_\eta|x|^2.
$
Therefore,
$$
    m_{2,k+1}
    =
    \mathbb E|F_\eta(\widetilde X_k)|^2+2\eta d
    \leq
    q_\eta m_{2,k}+2\eta d.
$$
Together with \eqref{eq:q-gap-moment}, this yields
\begin{equation}\label{eq:second-moment-pieta}
    \sup_{k\geq0}m_{2,k}
    \leq
    Cd.
\end{equation}

For any $a\in\mathbb R^d$ and
$\xi\sim\mathcal N(0,I_d)$,
$$
    \mathbb E|a+\sqrt{2\eta}\,\xi|^4
    =
    |a|^4
    +4\eta(d+2)|a|^2
    +4\eta^2d(d+2).
$$
Conditioning on $\widetilde X_k$, and using
\eqref{eq:F-contraction-moment} and
\eqref{eq:second-moment-pieta}, we obtain
$$
    m_{4,k+1}
    \leq
    q_\eta^2m_{4,k}
    +
    C\eta d^2.
$$
Since
$
1-q_\eta^2
\geq
1-q_\eta
\geq
3\alpha\eta/2,
$
it follows that
$$
    \sup_{k\geq0}m_{4,k}
    \leq
    \frac{C\eta d^2}{1-q_\eta^2}
    \leq
    Cd^2.
$$

Finally, for $R>0$, let
$
f_R(x)=|x|^4\wedge R.
$
Since $f_R$ is bounded and measurable, \eqref{2037} implies
$$
    \int_{\mathbb R^d} f_R(x)\,\pi_\eta(\dif x)
    =
    \lim_{k\to\infty}
    \mathbb E f_R(\widetilde X_k)
    \leq
    \sup_{k\geq0}m_{4,k}
    \leq
    Cd^2.
$$
Letting $R\to\infty$ and applying the monotone convergence theorem,
$$
    \int_{\mathbb R^d}|x|^4\,\pi_\eta(\dif x)
    \leq
    Cd^2.
$$
Since $V(x)=|x|^4+1$ and $d\geq1$, we conclude that
$$
    \pi_\eta(V)\leq Cd^2.
$$
\end{proof}

\begin{proof}[Proof of Lemma~\ref{873}]
Let $
    \mathcal F_k
    =
    \sigma(X_0,\xi_1,\ldots,\xi_k)
$
and set
\[
    Y_{k,i}
    =
    \left\langle
        \nabla^2\varphi_i(X_k),
        \xi_{k+1}\xi_{k+1}^{\top}-I_d
    \right\rangle_{\mathrm{HS}}.
\]
Since $\nabla^2\varphi_i(X_k)$ is
$\mathcal F_k$-measurable and $\xi_{k+1}$ is independent of
$\mathcal F_k$, we have
$
    \mathbb E[Y_{k,i}\mid\mathcal F_k]=0.
$ Hence, $(Y_{k,i})_{k\geq0}$ is a martingale-difference sequence,
and therefore
$
    \mathbb E\left|
        \sum_{k=0}^{n-1}Y_{k,i}
    \right|^2
    =
    \sum_{k=0}^{n-1}\mathbb EY_{k,i}^2.
$
For a symmetric matrix $H$ and a standard normal vector $\xi$,
$$
    \mathbb E
    \left[
        \left\langle
            H,\xi\xi^\top-I_d
        \right\rangle_{\mathrm{HS}}^2
    \right]
    =
    2\|H\|_{\mathrm{HS}}^2.
$$
Consequently, by \eqref{sobo2},
\begin{align*}
    \mathbb E[Y_{k,i}^2\mid\mathcal F_k]
    &=
    2\|\nabla^2\varphi_i(X_k)\|_{\mathrm{HS}}^2\leq
    2d\|\nabla^2\varphi_i(X_k)\|_{\mathrm{op}}^2
    \leq Cd.
\end{align*}
Summing over $k$ proves \eqref{461}.

Similarly, let
$
    Z_{k,i}
    =
    \left\langle
        \nabla^2\varphi_i(X_k),
        -\nabla U(X_k)\xi_{k+1}^{\top}
    \right\rangle_{\mathrm{HS}}.
$
Again, we have
$
    \mathbb E[Z_{k,i}\mid\mathcal F_k]=0,
$
and hence
$
    \mathbb E\left|
        \sum_{k=0}^{n-1}Z_{k,i}
    \right|^2
    =
    \sum_{k=0}^{n-1}\mathbb EZ_{k,i}^2.
$
Conditionally on $\mathcal F_k$,
\begin{align*}
    \mathbb E[Z_{k,i}^2\mid\mathcal F_k]
    &=
    \left|
        \nabla^2\varphi_i(X_k)\nabla U(X_k)
    \right|^2\leq
    C|\nabla U(X_k)|^2.
\end{align*}
By Lemma~\ref{Bound} and \eqref{AB1}, $\mathbb E|\nabla U(X_k)|^2\leq Cd.$ This proves \eqref{1547}. 		It remains to prove \eqref{1553}. For $t\in[0,1]$, write
		\begin{align*}
			A_{k,i}
			&=
			\left\langle
				\nabla^3\varphi_i(X_k),
				\xi_{k+1}^{\otimes3}
			\right\rangle_{\mathrm{HS}},\\
			B_{k,i}(t)
			&=
			\left\langle
				\nabla^3\varphi_i(X_k+t\Delta X_k)
				-\nabla^3\varphi_i(X_k),
				\xi_{k+1}^{\otimes3}
			\right\rangle_{\mathrm{HS}}.
		\end{align*}
		Then we have
		\begin{align}\label{eq:third-order-decomposition}
			\mathbb E
			\left|
				\sum_{k=0}^{n-1}
				\left\langle
					\nabla^3\varphi_i(X_k+t\Delta X_k),
					\xi_{k+1}^{\otimes3}
				\right\rangle_{\mathrm{HS}}
			\right|^2\leq
			2\mathbb E
			\left|
				\sum_{k=0}^{n-1}A_{k,i}
			\right|^2
			+
			2\mathbb E
			\left|
				\sum_{k=0}^{n-1}B_{k,i}(t)
			\right|^2.
		\end{align}
		
		Since $\nabla^3\varphi_i(X_k)$ is
		$\mathcal F_k$-measurable and $\xi_{k+1}$ is independent
		of $\mathcal F_k$, the vanishing of the third moments of a
		centered Gaussian vector implies that
		$
			\mathbb E[A_{k,i}\mid\mathcal F_k]=0.
		$
		Thus, $(A_{k,i})_{k\geq0}$ is a martingale-difference
		sequence, and hence
		\begin{align*}
			\mathbb E
			\left|
				\sum_{k=0}^{n-1}A_{k,i}
			\right|^2
			&=
			\sum_{k=0}^{n-1}\mathbb E|A_{k,i}|^2\leq
			C\sum_{k=0}^{n-1}
			\mathbb E|\xi_{k+1}|^6\leq
			Cnd^3,
		\end{align*}
		where we used \eqref{sobo2} and
		$\mathbb E|\xi_{k+1}|^6\leq Cd^3$.
		
		For the second term in
		\eqref{eq:third-order-decomposition}, the mean-value theorem
		and \eqref{sobo2} yield
		\begin{align*}
			|B_{k,i}(t)|
			&=
			\left|
				\int_0^t
				\left\langle
					\nabla^4\varphi_i(X_k+s\Delta X_k),
					\Delta X_k\otimes\xi_{k+1}^{\otimes3}
				\right\rangle_{\mathrm{HS}}
				\,\dif s
			\right|\leq
			C|\Delta X_k|\,|\xi_{k+1}|^3.
		\end{align*}
		Consequently,
		\begin{align*}
			\mathbb E
			\left|
				\sum_{k=0}^{n-1}B_{k,i}(t)
			\right|^2
			&\leq
			n\sum_{k=0}^{n-1}
			\mathbb E|B_{k,i}(t)|^2\leq
			Cn\sum_{k=0}^{n-1}
			\mathbb E
			\left[
				|\Delta X_k|^2|\xi_{k+1}|^6
			\right].
		\end{align*}
		Recall that
		$
			\Delta X_k
			=
			-\eta\nabla U(X_k)
			+\sqrt{2\eta}\,\xi_{k+1}.
		$
		Therefore,
		$$
			|\Delta X_k|^2
			\leq
			2\eta^2|\nabla U(X_k)|^2
			+
			4\eta|\xi_{k+1}|^2.
		$$
		Since $X_k$ is independent of $\xi_{k+1}$, it follows
		from \eqref{AB1}, Lemma~\ref{Bound}, and the Gaussian moment
		bounds that
		\begin{align*}
			\mathbb E
			\left[
				|\Delta X_k|^2|\xi_{k+1}|^6
			\right]
			&\leq
			C\eta^2
			\mathbb E|\nabla U(X_k)|^2
			\mathbb E|\xi_{k+1}|^6
			+
			C\eta\mathbb E|\xi_{k+1}|^8\leq
				C\eta^2d^4+C\eta d^4\leq
				C\eta d^4.
			\end{align*}
		Hence,
		$$
			\mathbb E
			\left|
				\sum_{k=0}^{n-1}B_{k,i}(t)
			\right|^2
			\leq
			C\eta n^2d^4.
		$$
		Combining the last two estimates with
		\eqref{eq:third-order-decomposition}, we obtain
		$$
			\mathbb E
			\left|
				\sum_{k=0}^{n-1}
				\left\langle
					\nabla^3\varphi_i(X_k+t\Delta X_k),
					\xi_{k+1}^{\otimes3}
				\right\rangle_{\mathrm{HS}}
			\right|^2
			\leq
			C\left(nd^3+\eta n^2d^4\right),
		$$
		which proves \eqref{1553}.
\end{proof}
	
    With the above preparations, we  now prove our main result.
    \begin{proof}[Proof of Theorem \ref{T2}]
	By the definition of $1$-Wasserstein distance, we have
	\begin{align}
    \mathcal{W}_1(\Sigma^{-1/2}W_n,\gamma)=&
		\mathcal{W}_1(\Sigma^{-1/2}(\mathcal{H}_n+\mathcal{R}_n),\gamma)\notag\\
        \leq& \sup_{f:\|f\|_{\rm{Lip}}\leq 1}|\mathbb{E}f(\Sigma^{-1/2}(\mathcal{H}_n+\mathcal{R}_n))-\gamma(f)|\notag\\
		\leq & \sup_{f:\|f\|_{\rm{Lip}}\leq 1}|\mathbb{E}f(\Sigma^{-1/2}\mathcal{H}_n)-\gamma(f)|\notag\\
		&+\sup_{f:\|f\|_{\rm{Lip}}\leq 1}|\mathbb{E}f(\Sigma^{-1/2}(\mathcal{H}_n+\mathcal{R}_n))-\mathbb{E}f(\Sigma^{-1/2}\mathcal{H}_n)|\notag\\
		\leq &  \mathcal{W}_1(\Sigma^{-1/2}\mathcal{H}_n,\gamma)+\mathbb{E}|\Sigma^{-1/2}\mathcal{R}_n|.
	\end{align}
Note the error estimates for
$\mathcal{W}_1(\Sigma^{-1/2}\mathcal{H}_n,\gamma)$  and $\mathbb{E}|\Sigma^{-1/2}\mathcal{R}_n|$  have been established in Propositions \ref{L21} and \ref{L22} respectively. Combining these two estimates completes the proof, as desired.
	\end{proof}

\section*{Acknowledgment}We thank  the anonymous referee  for his careful reading of the manuscript and  insightful comments and suggestions, which have significantly improved the quality of this paper. This work was partly  supported by the National Natural Science Foundation of China (Grant Nos. 12271475 and U23A2064).
		\bibliographystyle{plainnat}

\appendix
	
	\section{Proofs of Lemma \ref{strict} and Proposition \ref{prop1}}\label{B}
	
    To study the regularity of the solution of the Stein equation, we need to consider the derivative of $X_t^x$ in \eqref{LSDE} with respect to the initial value $x$, which is usually called the Jacobi flow. Let $v\in\R^d$. The Jacobi flow $\nabla_vX_t^x$ along the direction $v$ is defined by
	\begin{align}\label{jf1}
		\nabla_vX_t^x=\lim_{\e\to 0}\frac{X_t^{x+\e v}-X_t^x}{\e}, \ t\geq0.
	\end{align}
    Let $e_i$ be the $i$-th unit vector in $\R^d$, and define
    \begin{align*}
       \nabla X_t^x =(\nabla_{e_1} X_t^x,\nabla_{e_2} X_t^x,\cdots,\nabla_{e_d} X_t^x).
    \end{align*}

\begin{proof}[Proof of Lemma \ref{strict}]
Fix $u\in\mathbb R^d$, and set
$
    g_u(x)=\langle u,x-m_\pi\rangle,
   \varphi_u=\sum_{i=1}^d u_i\varphi_i.
$
Since $h_i(x)=x_i$ and $\pi(h_i)=(m_\pi)_i$, the linearity of
$\mathcal A$ gives
$$
    \mathcal A\varphi_u
    =
    \sum_{i=1}^d u_i\bigl(h_i-\pi(h_i)\bigr)
    =
    g_u.
$$
Moreover, we have
$
    \pi(g_u)=0,
    \nabla g_u=u.
$

We claim the following two standard facts that will be used below:

(i) \textit{Integration-by-parts identity.}
    Since
    $
        \pi(\dif x)\propto e^{-U(x)}\dif x
    $
    and
    $
        \mathcal A h
        =
        -\langle\nabla U,\nabla h\rangle+\Delta h,
    $
    integration by parts gives
    \begin{equation}\label{eq:dirichlet-identity}
        -\pi(f\mathcal A h)
        =
        \pi\bigl(\langle\nabla f,\nabla h\rangle\bigr).
    \end{equation}

    (ii) \textit{Poincar\'e inequality.}
    Since $\nabla^2U(x)\succeq\alpha I_d$, the strongly log-concave
    measure $\pi$ satisfies
    \begin{equation}\label{eq:poincare-pi}
        \pi\bigl((f-\pi(f))^2\bigr)
        \leq
        \frac{1}{\alpha}\pi\bigl(|\nabla f|^2\bigr).
    \end{equation}

We now turn to the bounds for $\Sigma$. By its definition,
\begin{align}\label{eq:Sigma-quadratic-form}
    u^\top\Sigma u
    &=
    2\sum_{i,j=1}^d
    u_i u_j
    \pi\bigl(
        \langle\nabla\varphi_i,\nabla\varphi_j\rangle
    \bigr)
    =
    2\pi\bigl(|\nabla\varphi_u|^2\bigr).
\end{align}

We first prove the upper bound. Taking $f=h=\varphi_u$ in
\eqref{eq:dirichlet-identity} and using
$\mathcal A\varphi_u=g_u$, we obtain
$$
    \pi\bigl(|\nabla\varphi_u|^2\bigr)
    =
    -\pi(\varphi_u g_u).
$$
Since $\pi(g_u)=0$,
$$
    \pi(\varphi_u g_u)
    =
    \pi\bigl((\varphi_u-\pi(\varphi_u))g_u\bigr).
$$
Hence, by the Cauchy--Schwarz inequality and
\eqref{eq:poincare-pi},
\begin{align*}
    \pi\bigl(|\nabla\varphi_u|^2\bigr)
    &\leq
    \pi\bigl((\varphi_u-\pi(\varphi_u))^2\bigr)^{1/2}
    \pi(g_u^2)^{1/2}\leq
    \alpha^{-1/2}
    \pi\bigl(|\nabla\varphi_u|^2\bigr)^{1/2}
    \pi(g_u^2)^{1/2}.
\end{align*}
Therefore,
$$
    \pi\bigl(|\nabla\varphi_u|^2\bigr)
    \leq
    \frac{1}{\alpha}\pi(g_u^2).
$$
Applying \eqref{eq:poincare-pi} to $g_u$ and using
$\pi(g_u)=0$ and $\nabla g_u=u$ give
$$
    \pi(g_u^2)
    \leq
    \frac{1}{\alpha}\pi\bigl(|\nabla g_u|^2\bigr)
    =
    \frac{|u|^2}{\alpha}.
$$
Consequently,
\begin{equation}\label{eq:Sigma-upper-pre}
    \pi\bigl(|\nabla\varphi_u|^2\bigr)
    \leq
    \frac{|u|^2}{\alpha^2}.
\end{equation}

We next prove the lower bound. We shall use the integration-by-parts
identity
\begin{equation}\label{eq:ibp-pi}
    \pi\bigl(f\langle v,\nabla U\rangle\bigr)
    =
    \pi\bigl(\langle v,\nabla f\rangle\bigr),
    \qquad v\in\mathbb R^d,
\end{equation}
which follows directly from
$
    \partial_j e^{-U}
    =
    -(\partial_jU)e^{-U}.
$
Taking $f=g_u$ and $v=u$ in \eqref{eq:ibp-pi}, and recalling that
$\nabla g_u=u$, we obtain
$$
    \pi\left(
        g_u\langle u,\nabla U\rangle
    \right)
    =
    |u|^2.
$$
Taking instead
$$
    f(x)=\langle u,\nabla U(x)\rangle,
    \qquad
    v=u
$$
gives
$$
    \pi\left(
        \langle u,\nabla U\rangle^2
    \right)
    =
    \pi\left(
        u^\top\nabla^2Uu
    \right)
    \leq
    \beta|u|^2,
$$
where we used $\nabla^2U\preceq\beta I_d$.
Thus, by the Cauchy--Schwarz inequality,
$$
    |u|^4
    \leq
    \pi(g_u^2)
    \pi\left(
        \langle u,\nabla U\rangle^2
    \right)
    \leq
    \beta|u|^2\pi(g_u^2).
$$
Hence, for $u\neq0$,
\begin{equation}\label{eq:gu-lower}
    \pi(g_u^2)
    \geq
    \frac{|u|^2}{\beta}.
\end{equation}

On the other hand, taking $f=g_u$ and $h=\varphi_u$ in
\eqref{eq:dirichlet-identity}, and using
$\mathcal A\varphi_u=g_u$, we obtain
$$
    \pi(g_u^2)
    =
    -\pi\bigl(
        \langle\nabla g_u,\nabla\varphi_u\rangle
    \bigr).
$$
Since $\nabla g_u=u$, the Cauchy--Schwarz inequality yields
$$
    \pi(g_u^2)
    \leq
    |u|\,
    \pi\bigl(|\nabla\varphi_u|^2\bigr)^{1/2}.
$$
Combining this with \eqref{eq:gu-lower}, we obtain
$$
    \pi\bigl(|\nabla\varphi_u|^2\bigr)
    \geq
    \frac{\pi(g_u^2)^2}{|u|^2}
    \geq
    \frac{|u|^2}{\beta^2}.
$$
The same conclusion is trivial when $u=0$.

Combining this lower bound with \eqref{eq:Sigma-upper-pre} and
\eqref{eq:Sigma-quadratic-form}, we conclude that
$$
    \frac{2}{\beta^2}|u|^2
    \leq
    u^\top\Sigma u
    \leq
    \frac{2}{\alpha^2}|u|^2,
    \qquad u\in\mathbb R^d.
$$
Since $u$ is arbitrary,
$
    {2}I_d/{\beta^2}
    \preceq
    \Sigma
    \preceq
  {2}I_d/{\alpha^2}.
$
In particular, $\Sigma$ is positive definite.
\end{proof}

	Turn to Proposition \ref{prop1}. As in \eqref{jf1}, for $v_1, v_2, v_3, v_4\in\R^d$, we  define
	
	\begin{align*}
		\nabla_{v_2}\nabla_{v_1}X_t^x=\lim_{\e\to 0}\frac{\nabla_{v_1}X_t^{x+\e v_2}-\nabla_{v_1}X_t^x}{\e}, \ t\geq0,
	\end{align*}
	\begin{align*}
		\nabla_{v_3}\nabla_{v_2}\nabla_{v_1}X_t^x=\lim_{\e\to 0}\frac{\nabla_{v_2}\nabla_{v_1}X_t^{x+\e v_3}-\nabla_{v_2}\nabla_{v_1}X_t^x}{\e}, \ t\geq0
	\end{align*}
	and
	\begin{align*}
		\nabla_{v_4}\nabla_{v_3}\nabla_{v_2}\nabla_{v_1}X_t^x=\lim_{\e\to 0}\frac{\nabla_{v_3}\nabla_{v_2}\nabla_{v_1}X_t^{x+\e v_4}-\nabla_{v_3}\nabla_{v_2}\nabla_{v_1}X_t^x}{\e}, \ t\geq0.
	\end{align*}

	\begin{lemma}\label{jm}
		Assume that the function $U$ in \eqref{LSDE} satisfies \textbf{Assumption} \ref{Ass1}. Then there exists some positive constant $C$ such that 
		\begin{align*}
			\E[|\nabla_{v_1}X_t^x|^{16}]\leq &|v_1|^{16}e^{-16\alpha t},\\
			\E[|\nabla_{v_2}\nabla_{v_1}X_t^x|^{8}]\leq &C |v_1|^8|v_2|^8e^{-4\alpha t},\\
			\E[|\nabla_{v_3}\nabla_{v_2}\nabla_{v_1}X_t^x|^{4}] \leq &C|v_1|^4|v_2|^4|v_3|^4e^{-2\alpha t},\\
			\E[|\nabla_{v_4}\nabla_{v_3}\nabla_{v_2}\nabla_{v_1}X_t^x|^{2}] \leq &C|v_1|^2|v_2|^2|v_3|^2|v_4|^2e^{-\alpha t}.
		\end{align*}
	\end{lemma}

    For the sake of reading, the proof of this lemma is deferred until after the proof of Proposition \ref{prop1}.
	
	\begin{proof} [Proof of Proposition \ref{prop1}]
		Since $\nabla g$ is bounded, $g$ is Lipschitz and has at
	most linear growth. We first verify that \eqref{solu} is well
	defined. Let $X_t^x$ and $X_t^y$ be synchronously coupled.
	By the strong convexity of $U$,
	$$
		|X_t^x-X_t^y|
		\leq
		e^{-\alpha t}|x-y|.
	$$
	Since $\pi$ is invariant, it follows that
	\begin{align*}
		|P_tg(x)-\pi(g)|
		&\leq
		\int_{\mathbb R^d}
		\mathbb E
		|g(X_t^x)-g(X_t^y)|
		\,\pi(\dif y)\leq
		\|\nabla g\|_{\mathrm{op}}
		e^{-\alpha t}
		\int_{\mathbb R^d}|x-y|\,\pi(\dif y).
	\end{align*}
	Thus, the integral in \eqref{solu} is absolutely convergent.
	Moreover, the semigroup identity gives
	$$
		\mathcal Af_g=g-\pi(g).
	$$
	We now prove \eqref{sobo1} and \eqref{sobo2}.

		(i)  Note that
        \begin{align*}
			f_g(x) = -\int_{0}^{\infty} \mathbb{E}\big[g(X_t^x) - \pi(g)\big] \, \dif t.
		\end{align*}
        
         Then, for any $v\in\R^d$, 
		\begin{align}\label{1910}
			|\nabla_v f_g(x)|=&\left|-\int_0^\infty\E[\langle\nabla g(X_t^x),\nabla_{v}X_t^x\rangle]\dif t\right|
			\leq\int_0^\infty\E[|\nabla g||\nabla_{v}X_t^x|]\dif t
			\leq\frac{C}{\alpha}|v|,
		\end{align}
		where the last inequality follows from Lemma \ref{jm}. So we obtain
		\begin{align}\label{1911}
			\|\nabla f_g\|=&\sup_x\|\nabla f_g(x)\|_{\rm op}
			=\sup_x\sup_{|v|=1}|\nabla_v f_g(x)|
			\leq\frac{C}{\alpha}.
		\end{align}
		
		(ii) For any $v_1, v_2, v_3\in\R^d$, we easily have
		\begin{align*}
			&\nabla_{v_2}\nabla_{v_1}f_g(x)\\
            =&-\int_0^\infty\nabla_{v_2}\E[\langle\nabla g(X_t^x),\nabla_{v_1}X_t^x\rangle]\dif t\\
			=&-\int_0^\infty\E[\langle\nabla g(X_t^x),\nabla_{v_2}\nabla_{v_1}X_t^x\rangle]\dif t-\int_0^\infty\E[\langle\nabla^2 g(X_t^x)\nabla_{v_2}X_t^x,\nabla_{v_1}X_t^x\rangle]\dif t
		\end{align*}	
        and
        \begin{align*}
        \nabla_{v_3}\nabla_{v_2}\nabla_{v_1}f_g(x)=&
			-\int_0^\infty\E[\langle\nabla g(X_t^x),\nabla_{v_3}\nabla_{v_2}\nabla_{v_1}X_t^x\rangle]\dif t\\
			&-\int_0^\infty\E[\langle\nabla^2 g(X_t^x)\nabla_{v_3}X_t^x,\nabla_{v_2}\nabla_{v_1}X_t^x\rangle]\dif t\\
			&-\int_0^\infty\E[\langle\nabla^2 g(X_t^x)\nabla_{v_3}\nabla_{v_2}X_t^x,\nabla_{v_1}X_t^x\rangle]\dif t\\
			&-\int_0^\infty\E[\langle\nabla^2 g(X_t^x)\nabla_{v_2}X_t^x,\nabla_{v_3}\nabla_{v_1}X_t^x\rangle]\dif t\\
			&-\int_0^\infty\E\langle\nabla^3 g(X_t^x),\, \nabla_{v_1}X_t^x\otimes\nabla_{v_2}X_t^x\otimes\nabla_{v_3}X_t^x\rangle_{\mathrm{HS}}\dif t.\end{align*}
		
		Using an argument similar to \eqref{1910} and \eqref{1911} together with Lemma \ref{jm}, the H\"older inequality, and the condition $\|\nabla^k  g\|_{\mathrm{op}}\leq C$ for $k=1,2,3$, we derive  \eqref{sobo1}.		
		
		Next, consider the coordinate map $g=h_i$. Since $\nabla h_i=e_i$ and $\nabla^kh_i=0, k=2,3,4,$
the bound \eqref{sobo1} gives
$$
	\|\nabla^k\varphi_i\|_{\mathrm{op}}\leq C,
	\qquad k=1,2,3.
$$
It remains to prove the bound for $k=4$. Actually, 
       $\nabla^k h_i=0$ for $k=2,3,4$, and hence we have
		\begin{align*}
			\nabla_{v_4} \nabla_{v_3}\nabla_{v_2}\nabla_{v_1}\vp_i(x)=&
			-\int_0^\infty\E[\langle\nabla h_i(X_t^x),\nabla_{v_4}\nabla_{v_3}\nabla_{v_2}\nabla_{v_1}X_t^x\rangle]\dif t.\end{align*}
		Applying Lemma \ref{jm} and the H\"older inequality again, we obtain \eqref{sobo2} as desired.
	\end{proof}

	\begin{proof}[Proof of Lemma \ref{jm}]
		(i) It follows from \eqref{LSDE} that
		\begin{align*}
			X_t^x=x+\int_0^t -\nabla U(X_s^x)\dif s+\sqrt{2} B_t.
		\end{align*}
		Combining this with \eqref{jf1},  for any $v_1\in \R^d$, one has $\nabla_{v_1}X_0^x=v_1$ and
		\begin{align*}
			\dif\nabla_{v_1}X_t^x=-\nabla^2 U(X_t^x) \nabla_{v_1}X_t^x\dif t.
		\end{align*}
		
		Using It\^{o}'s formula for the function $f(x)=|x|^{16}$, one has
		\begin{align*}
			\dif|\nabla_{v_1}X_t^x|^{16}=&\langle 16|\nabla_{v_1}X_t^x|^{14}\nabla_{v_1}X_t^x,- \nabla^2U(X_t^x)\nabla_{v_1}X_t^x\rangle\dif t
			\leq-16\alpha|\nabla_{v_1}X_t^x|^{16}\dif t,
		\end{align*}
		which implies that
		\begin{align}\label{m1}
			\E[|\nabla_{v_1}X_t^x|^{16}]\leq |v_1|^{16}e^{-16\alpha t}.
		\end{align}
		
		(ii) For any $v_1, v_2\in \R^d$, we know $\nabla_{v_2}\nabla_{v_1}X_0^x=0$ and
		\begin{align*}
			\dif \nabla_{v_2}\nabla_{v_1}X_t^x=\big(-\nabla ^3U(X_t^x)\nabla_{v_2}X_t^x\nabla_{v_1}X_t^x-\nabla^2U(X_t^x)\nabla_{v_2}\nabla_{v_1}X_t^x\big)\dif t.
		\end{align*}
		
		Let $f(x)=|x|^8$. Using It\^{o}'s formula, we have
		\begin{align*}
			&\dif |\nabla_{v_2}\nabla_{v_1}X_t^x|^8\\
			=&\langle 8|\nabla_{v_2}\nabla_{v_1}X_t^x|^{6}\nabla_{v_2}\nabla_{v_1}X_t^x,-\nabla^3U(X_t^x)\nabla_{v_2}X_t^x\nabla_{v_1}X_t^x-\nabla^2U(X_t^x)\nabla_{v_2}\nabla_{v_1}X_t^x\rangle\dif t\\
			\leq&\big(-8\alpha|\nabla_{v_2}\nabla_{v_1}X_t^x|^{8}+8\|\nabla^3U\|_{\mathrm{op}}|\nabla_{v_2}\nabla_{v_1}X_t^x|^{7}|\nabla_{v_1}X_t^x||\nabla_{v_2}X_t^x|\big)\dif t.
		\end{align*}
		
		By Young's inequality, for any positive constant $c$, one has
		
		\begin{align*}
			c|\nabla_{v_2}\nabla_{v_1}X_t^x|^{7}\frac{1}{c}|\nabla_{v_1}X_t^x||\nabla_{v_2}X_t^x|
			\leq&\frac{c^{\frac{8}{7}}|\nabla_{v_2}\nabla_{v_1}X_t^x|^{8}}{\frac{8}{7}}+\frac{|\nabla_{v_1}X_t^x|^8|\nabla_{v_2}X_t^x|^8}{c^8}.
		\end{align*}
		
		Choosing $c$ small enough such that $7c^\frac{8}{7}\|\nabla^3U\|_{\mathrm{op}}<4\alpha$, we obtain
		\begin{align*}
			\dif |\nabla_{v_2}\nabla_{v_1}X_t^x|^8\leq \big(-4\alpha|\nabla_{v_2}\nabla_{v_1}X_t^x|^{8}+K|\nabla_{v_1}X_t^x|^8|\nabla_{v_2}X_t^x|^8\big)\dif t,
		\end{align*}
		where the constant $K$ only depends on $\alpha$ and $M$.

		It follows from the comparison theorem (see \cite[Chapter~V, Theorem~54]{Plotter2003Stochastic}) that
		\begin{align}\label{m2}
				&\E[|\nabla_{v_2}\nabla_{v_1}X_t^x|^8]\leq Ce^{-4\alpha t}\int_0^t    e^{4\alpha s}\E[|\nabla_{v_1}X_s^x|^8|\nabla_{v_2}X_s^x|^8]\dif s\nonumber\\
			\leq&C|v_1|^8|v_2|^8e^{-4\alpha t}\int_0^t    e^{-12\alpha s}\dif s
			\leq C|v_1|^8|v_2|^8e^{-4\alpha t},
		\end{align}
		where we used the moment estimate for the first-order Jacobi flow in the second inequality.
		
		(iii)  For $v_1,v_2,v_3\in\R^d$, one has $\nabla_{v_3}\nabla_{v_2}\nabla_{v_1}X_0^x=0$ and $\dif \nabla_{v_3}\nabla_{v_2}\nabla_{v_1}X_t^x=G_t\dif t,$ where
		\begin{align*}
			G_t=&-\nabla^4 U (X_t^{x})   \nabla_{v_3} X_t^{x}  \nabla_{v_2} X_t^{x}  \nabla_{v_1} X_t^{x}
			-\nabla^2 U (X_t^{x}) \nabla_{v_3} \nabla_{v_2}  \nabla_{v_1} X_t^{x}  \\
			& -\nabla^3U(X_t^{x})[\nabla_{v_3}  \nabla_{v_2} X_t^{x}\nabla_{v_1} X_t^{x}+\nabla_{v_3}  \nabla_{v_1} X_t^{x}\nabla_{v_2} X_t^{x}+\nabla_{v_1}  \nabla_{v_2} X_t^{x}\nabla_{v_3} X_t^{x}].
		\end{align*}
		
		Using It\^{o}'s formula for $f(x)=|x|^4$, we get
		\begin{align*}
			&\dif |\nabla_{v_3}\nabla_{v_2}\nabla_{v_1} X_t^{x}|^4\\
			=&\langle 4|\nabla_{v_3}\nabla_{v_2}\nabla_{v_1} X_t^{x}|^2\nabla_{v_3}\nabla_{v_2}\nabla_{v_1} X_t^{x},G_t\rangle\dif t\\
			\leq&\big[-4\alpha|\nabla_{v_3}\nabla_{v_2}\nabla_{v_1} X_t^{x}|^4+4\|\nabla^4U\|_{\mathrm{op}}|\nabla_{v_3}\nabla_{v_2}\nabla_{v_1} X_t^{x}|^3|\nabla_{v_1}X_t^{x}||\nabla_{v_2}X_t^{x}||\nabla_{v_3}X_t^{x}|\\
			&+4\|\nabla^3U\|_{\mathrm{op}}|\nabla_{v_3}\nabla_{v_2}\nabla_{v_1} X_t^{x}|^3
			(|\nabla_{v_1}  \nabla_{v_2} X_t^{x}||\nabla_{v_3} X_t^{x}|\\
			&+|\nabla_{v_1}  \nabla_{v_3} X_t^{x}||\nabla_{v_2} X_t^{x}|+|\nabla_{v_2}  \nabla_{v_3} X_t^{x}||\nabla_{v_1} X_t^{x}|)\big]\dif t.
		\end{align*}
		
		By Young's inequality, we have
		\begin{align*}
			\dif |\nabla_{v_3}\nabla_{v_2}\nabla_{v_1} X_t^{x}|^4
			\leq (-2\alpha |\nabla_{v_3}\nabla_{v_2}\nabla_{v_1} X_t^{x}|^4+F_t)\dif t,
		\end{align*}
		where
		\begin{align*}
			F_t=&K_1|\nabla_{v_3} X_t^{x}|^4|\nabla_{v_2} X_t^{x}|^4|\nabla_{v_1} X_t^{x}|^4\\
			&+K_2(|\nabla_{v_1}  \nabla_{v_2} X_t^{x}|^4|\nabla_{v_3} X_t^{x}|^4+|\nabla_{v_1}  \nabla_{v_3} X_t^{x}|^4|\nabla_{v_2} X_t^{x}|^4+|\nabla_{v_2}  \nabla_{v_3} X_t^{x}|^4|\nabla_{v_1} X_t^{x}|^4),
		\end{align*}
		$K_1$ and $K_2$  depend on $\alpha$ and $M$.
		
		It follows from \eqref{m1}, \eqref{m2} and the comparison theorem that
		\begin{align}\label{m3}
			\E[|\nabla_{v_3}\nabla_{v_2}\nabla_{v_1} X_t^{x}|^4]
			\leq &e^{-2\alpha t}\int_0^t    e^{2\alpha s}\E[F_s]\dif s\nonumber\\
			\leq & e^{-2\alpha t}\int_0^t    e^{2\alpha s}
			(K_1|v_1|^4|v_2|^4|v_3|^4e^{-12\alpha s}+K_2|v_1|^4|v_2|^4|v_3|^4e^{-6\alpha s})\dif s\nonumber\\
			\leq& C|v_1|^4|v_2|^4|v_3|^4e^{-2\alpha t}.
		\end{align}

		(iv) Similarly, combining \eqref{m1}-\eqref{m3},  It\^{o}'s formula, Young's inequality and the comparison theorem, we can get
		\begin{align*}
			\E[|\nabla_{v_4}\nabla_{v_3}\nabla_{v_2}\nabla_{v_1}X_t^x|^{2}] \leq &C|v_1|^2|v_2|^2|v_3|^2|v_4|^2e^{-\alpha t}.
		\end{align*}
		 The proof is complete.
	\end{proof}
    
\end{document}